\documentclass[12pt,a4paper,twoside]{article}
\usepackage{a4wide, amsmath,amsbsy,amsfonts,amssymb, 
stmaryrd,amsthm,mathrsfs,graphicx, amscd, tikz-cd, cancel}
\usepackage[normalem]{ulem}
\usetikzlibrary{matrix,arrows,decorations.pathmorphing}
\usepackage[active]{srcltx}
\usepackage[
	hypertexnames=false,
	hyperindex,
	pagebackref,
	pdftex,
	breaklinks=true,
	bookmarks=false,
	colorlinks,
	linkcolor=blue,
	citecolor=red,
	urlcolor=red,
]{hyperref}
\DeclareMathSymbol{\Alpha}{\mathalpha}{operators}{"41}
\DeclareMathSymbol{\Beta}{\mathalpha}{operators}{"42}
\DeclareMathSymbol{\Epsilon}{\mathalpha}{operators}{"45}
\DeclareMathSymbol{\Zeta}{\mathalpha}{operators}{"5A}
\DeclareMathSymbol{\Eta}{\mathalpha}{operators}{"48}
\DeclareMathSymbol{\Iota}{\mathalpha}{operators}{"49}
\DeclareMathSymbol{\Kappa}{\mathalpha}{operators}{"4B}
\DeclareMathSymbol{\Mu}{\mathalpha}{operators}{"4D}
\DeclareMathSymbol{\Nu}{\mathalpha}{operators}{"4E}
\DeclareMathSymbol{\Omicron}{\mathalpha}{operators}{"4F}
\DeclareMathSymbol{\Rho}{\mathalpha}{operators}{"50}
\DeclareMathSymbol{\Tau}{\mathalpha}{operators}{"54}
\DeclareMathSymbol{\Chi}{\mathalpha}{operators}{"58}
\DeclareMathSymbol{\omicron}{\mathord}{letters}{"6F}

\newcommand\ZZ{{\hat{\mathbb Z}}}
\newcommand\Z{{\mathbb Z}}
\newcommand\PP{{\mathbb P}}
\newcommand\Q{{\mathbb Q}}
\newcommand\F{{\mathbb F}}

\newcommand\C{{\mathbb C}}
\newcommand\N{{\mathbb N}}
\newcommand\cL{{\mathscr L}}
\newcommand\hL{\hat{\mathscr L}}
\newcommand\cP{{\mathscr P}}
\newcommand\ccP{\check{\Pi}}
\newcommand\ra{\rightarrow}

\newcommand\ilim{\lim\limits_{\longleftarrow}\,}

\newcommand\Sp{\operatorname{Sp}}

\newcommand\aut{\operatorname{Aut}}
\newcommand\out{\operatorname{Out}}
\newcommand\inn{\operatorname{inn}}
\newcommand\Inn{\operatorname{Inn}}

\renewcommand\dim{\operatorname{dim}}

\newcommand\hookra{\hookrightarrow}
\newcommand\tura{\twoheadrightarrow}
\newcommand\da{\downarrow}

\newcommand\dd{\partial}

\newcommand{\Id}{\operatorname{Id}}
 
\newcommand{\sym}{\operatorname{Sym}}
\newcommand\sr{\stackrel}
\newcommand\st{\scriptstyle}
\newcommand\sst{\scriptscriptstyle}
\newcommand\cGG{\check{\GG}}
\newcommand\hGG{\widehat{\GG}}
\newcommand\UU{\Upsilon}
\newcommand\hUU{\widehat{\Upsilon}}
\newcommand\cUU{\check{\Upsilon}}

\newcommand\cK{\mathscr K}

\newcommand\hP{\widehat{\Pi}}
\newcommand\hp{\hat{\pi}}

\newcommand\tGG{\tilde{\GG}}

\newcommand\ssm{\smallsetminus}
\newcommand\ol{\overline}

\newcommand\cM{{\mathcal M}}

\newcommand\cH{{\cal H}}

\newcommand\cC{{\mathscr C}}

\newcommand\cF{{\mathscr F}}

\newcommand\cS{{\mathscr S}}
\newcommand\cU{{\mathscr U}}

\newcommand\GG{\Gamma}
\newcommand\ld{\lambda}
\newcommand\Ld{\Lambda}
\newcommand\wh{\widehat}

\newcommand\td{\tilde}
\newcommand\sg{\sigma}
\newcommand\Sg{\Sigma}
\newcommand\gm{\gamma}

\def\co{\colon\thinspace}

\newtheorem{theorem}{Theorem}[section]
\newtheorem{corollary}[theorem]{Corollary}
\newtheorem{proposition}[theorem]{Proposition}
\newtheorem{lemma}[theorem]{Lemma}

\newtheorem{question}[theorem]{Question}

\theoremstyle{definition}
\newtheorem{definition}[theorem]{Definition}   
\newtheorem{remark}[theorem]{Remark}  

\begin{document}

\title{Congruence topologies on the mapping class group}
\author{Marco Boggi}\maketitle

\begin{abstract}Let $\Gamma(S)$ be the pure mapping class group of a connected orientable surface $S$ of negative Euler characteristic.
For ${\mathscr C}$ a class of finite groups, let $\hat{\pi}_1(S)^{\mathscr C}$ be the pro-${\mathscr C}$ completion of the fundamental group of $S$.
The \emph{${\mathscr C}$-congruence completion $\check{\Gamma}(S)^{\mathscr C}$ of $\Gamma(S)$} is the profinite completion induced by the 
embedding $\Gamma(S)\hookrightarrow{\operatorname{Out}}(\hat{\pi}_1(S)^{\mathscr C})$. In this paper, we begin a systematic study of 
such completions for different ${\mathscr C}$. We show that the combinatorial structure of the profinite groups $\check{\Gamma}(S)^{\mathscr C}$ 
closely resemble that of $\Gamma(S)$. 
A fundamental question is how ${\mathscr C}$-congruence completions compare with pro-${\mathscr C}$ completions. Even though,
in general (e.g.\ for ${\mathscr C}$ the class of finite solvable groups), $\check{\Gamma}(S)^{\mathscr C}$ is not even virtually a pro-${\mathscr C}$ group,
we show that, for ${\mathbb Z}/2\in{\mathscr C}$, $g(S)\leq 2$ and $S$ open, there is a natural epimorphism from the ${\mathscr C}$-congruence 
completion $\check{\Gamma}(S)(2)^{\mathscr C}$ of the abelian level of order $2$ to its pro-${\mathscr C}$ completion 
$\widehat{\Gamma}(S)(2)^{\mathscr C}$. In particular, this is an isomorphism for the class of finite groups and for 
the class of $2$-groups. Moreover, in these two cases, the result also holds for a closed surface.
\newline

\end{abstract}

\section{Introduction}\label{intro}
Let $S$ be a connected surface of negative Euler characteristic. We denote by $\GG(S)$ its \emph{pure} 
mapping class group, that it is to say the group of isotopy classes of orientation preserving homeomorphism (or diffeomorphism) of $S$ 
in itself which do not permute punctures or boundary components of $S$.

For $S_g$ a closed orientable Riemann surface of genus $g$ and $\{P_1,\ldots, P_{n}\}$
a set of distinct points on $S_g$ such that $2g-2+n>0$, we then let $S_{g,n}:=S_g\ssm\{P_1,\ldots, P_{n}\}$ and denote 
$\GG(S_{g,n})$ by $\GG_{g,n}$. Let $\Pi_{g,n}:=\pi_1(S_{g,n},P_{n+1})$ be the fundamental group of $S_{g,n}$ based at $P_{n+1}$. 
There is a faithful representation:
\[\rho_{g,n}\co \GG_{g,n}\hookra\out(\Pi_{g,n}).\]

A \emph{level} of $\GG_{g,n}$ is just a finite index subgroup $H<\GG_{g,n}$.
A characteristic finite index subgroup $K$ of $\Pi_{g,n}$ determines the {\it geometric level}
$\GG_K$, defined to be the kernel of the induced representation:
\[\rho_K\co \GG_{g,n}\to\out(\Pi_{g,n}/K).\]
The abelian levels $\GG(m)$ of order $m$, for $m\geq 2$, are a particular case and are defined by 
the kernel of the natural representation:
\[\rho_{(m)}\co \GG_{g,n}\to\Sp(H_1(S_{g},\Z/m)).\]

\begin{definition}\emph{A class of finite groups} is a full 
subcategory $\cC$ of the category of finite groups which is closed under taking subgroups, 
homomorphic images and extensions (meaning that a short exact sequence of finite groups is 
in $\cC$ whenever its exterior terms are). We always assume that $\cC$ contains a nontrivial group. We denote by $\cF$ the class of all finite groups, 
by $\cS$ the class of finite solvable groups and by $(p)$, for $p>1$ a prime number, the class of $p$-groups.
\end{definition}

We say that a subgroup $H$ of a group $G$ is \emph{$\cC$-open} if it contains a normal subgroup $N$ of $G$ 
such that the quotient group $G/N$ belongs to $\cC$. In this case, we write $H\leq_\cC G$ and $N\lhd_\cC G$, 
respectively. For a finitely generated group $G$, a fundamental system of neighborhoods of the identity
for the \emph{pro-$\cC$ topology} is given by any cofinal system of $\cC$-open characteristic subgroups. 
The \emph{pro-$\cC$ completion} $\wh{G}^\cC$ is the completion of the group $G$ with respect to its pro-$\cC$ topology. We denote $\wh{G}^\cF$
simply by $\wh{G}$.

The pro-$\cC$ topology on the fundamental group $\Pi_{g,n}$ determines a profinite topology on the mapping class
group $\GG_{g,n}$ by associating to the set of characteristic $\cC$-open subgroups of $\Pi_{g,n}$ the levels 
$\{\GG_K\}_{K\lhd_\cC\Pi_{g,n}}$, which form a fundamental system of neighborhoods of the identity for a profinite topology
on the mapping class group $\GG_{g,n}$. 

We say that a subgroup of $\GG_{g,n}$ is \emph{$\cC$-congruence open} if it contains some level $\GG_K$ for 
$K\lhd_\cC\Pi_{g,n}$. We call the topology thus defined the \emph{$\cC$-congruence topology on} $\GG_{g,n}$ and
denote by $\cGG_{g,n}^\cC$ the corresponding completion. Also, for any subgroup $H$ of $\GG_{g,n}$, the \emph{$\cC$-congruence 
topology} on $H$ is the profinite topology whose fundamental system of neighborhoods of the identity is the set of normal finite index subgroups 
$\{H\cap\GG_K\}_{K\lhd_\cC\Pi_{g,n}}$. The \emph{$\cC$-congruence completion} $\check{H}^\cC$ of $H$ is the associated profinite completion.

The classical congruence subgroup problem asks whether the class of all finite groups of $\Pi_{g,n}$ induces on $\GG_{g,n}$
the full profinite topology, or, equivalently, if the set of geometric levels is cofinal in the set of finite index subgroups of $\GG_{g,n}$, 
ordered by inclusion. This problem is settled, affirmatively, only for $g\leq 2$ (cf.\ \cite{hyp} and (ii) of Theorem~\ref{main}).

Clearly, it is an interesting problem to understand the topology of $\cGG_{g,n}^\cC$ for more general classes of finite groups.  
By Lemma~4.5.5 in \cite{RZ}, the $p$-congruence completion $\cGG_{g,n}^{(p)}$ of $\GG_{g,n}$ is virtually a pro-$p$ group,
that is to say, it contains an open subgroup which is pro-$p$. So it makes sense to ask whether $\cGG_{g,n}^{(p)}$ is virtually the pro-$p$ completion
of $\GG_{g,n}$. However, in general, it is not true that the profinite group $\cGG_{g,n}^\cC$ contains an open subgroup which is a pro-$\cC$ group.
In fact, in Proposition~\ref{novirtual}, we will show that, for $\cC=\cS$, the class of finite solvable groups, $\cGG_{g,n}^\cS$ is not virtually pro-$\cS$. 
This motivates the following definition:

\begin{definition}\label{weak_strong}Let $\cC$ be a class of finite groups:
\begin{enumerate}
\item The mapping class group $\GG_{g,n}$ satisfies the \emph{$\cC$-congruence subgroup property} if $\cGG_{g,n}^\cC$ is virtually the
pro-$\cC$ completion of $\GG_{g,n}$, that is to say, for some $\cC$-congruence open subgroup $U$ of $\GG_{g,n}$, denoting by $\check{U}^\cC$ 
its $\cC$-congruence completion, we have $\check{U}^\cC\equiv\hat{U}^\cC$.
\item The mapping class group $\GG_{g,n}$ satisfies the \emph{weak $\cC$-congruence subgroup property} if there is a $\cC$-congruence open
subgroup $U$ of $\GG_{g,n}$ which is residually a $\cC$-group and such that the $\cC$-congruence topology induces on $U$ a topology which is 
finer than the pro-$\cC$ topology. Equivalently:  there is a natural epimorphism $\check{U}^\cC\tura\hat{U}^\cC$.
\end{enumerate}\end{definition}

\begin{remark}In Definition~\ref{weak_strong}, (i) implies (ii). If $\cGG_{g,n}^\cC$ is virtually pro-$\cC$, (i) and (ii) are equivalent. In (ii), we ask that $U$ is
residually a $\cC$-group in order to avoid that the property may hold trivially. For instance, if we let $U=\GG_{g,n}$, then the pro-$\cS$ completion of $U$ 
is trivial for $g\geq 3$ (cf.\ Theorem~5.2 in \cite{FM}). Instead, for a prime $p\geq 0$, it is well known (cf.\ Section~\ref{standard}) that the abelian level 
$\GG(p)$ of $\GG_{g,n}$ is residually a $p$-group and so residually a $\cC$-group for any class of finite groups $\cC$ such that $\Z/p\in\cC$.
\end{remark}

In Corollary~\ref{congruence_0}, we will show that the weak $\cC$-congruence subgroup property is indeed satisfied in genus $0$
for any class of finite groups $\cC$. The next cases to consider then are genus $1$ and $2$, which are both subsumed under the notion
of \emph{hyperelliptic mapping class group}, which we define next.

For $g\geq 2$, let us fix a hyperelliptic involution $\iota$ in the mapping class group $\GG_g$. The hyperelliptic mapping class group
$\UU_g$ is the centralizer of $\iota$ in $\GG_g$. For $n>0$, we then define the hyperelliptic mapping class group $\UU_{g,n}$
to be the inverse image of $\UU_g$ via the natural epimorphism $\GG_{g,n}\to\GG_g$. Thus, for $g=2$, there holds $\UU_{2,n}=\GG_{2,n}$. 
For $g=1$ and $n\geq 1$, we let instead $\UU_{1,n}:=\GG_{1,n}$. From the definition, it follows that there is a natural faithful representation:
\[\rho_{g,n}\co \UU_{g,n}\hookra\out(\Pi_{g,n}).\]

The definition of the hyperelliptic mapping class group can be justified geometrically as follows.
Let $\cM_{g,n}$, for $2g-2+n>0$, be the moduli stack of $n$-pointed, genus $g$, smooth algebraic 
complex curves. It is a smooth connected Deligne-Mumford stack (briefly {\it D-M stack}) over $\C$ of
dimension $3g-3+n$, whose associated underlying complex 
analytic and topological {\'e}tale groupoids, we both denote by $\cM_{g,n}$ as well.

For stacks of the kind of $\cM_{g,n}$, the fundamental group can be described in a 
simple way. In fact, $\cM_{g,n}$ has a universal cover $T_{g,n}$ in the category of analytic
manifolds. The fundamental group $\pi_1(\cM_{g,n},[C])$ is then identified with the deck 
transformation group of the universal cover $T_{g,n}\to\cM_{g,n}$.

The choice of a lift of a point $[C]\in\cM_{g,n}$ to $T_{g,n}$ and of a diffeomorphism 
$S_{g,n}\to C\ssm\{\mbox{marked points}\}$ identifies the mapping class group $\GG_{g,n}$ 
with $\pi_1(\cM_{g,n},[C])$. The representation:
\[\rho_{g,n}\co\pi_1(\cM_{g,n},[C])\to\out(\Pi_{g,n}),\]
induced by the identification of $\GG_{g,n}$ with $\pi_1(\cM_{g,n},[C])$,
is equivalent to the universal monodromy representation associated to the short exact sequence 
of topological fundamental groups, determined by the $n$-punctured, genus $g$, universal curve ${\mathcal C}_{g,n}\to\cM_{g,n}$.

Let us assume that the point $[C]\in\cM_{g,n}$ parameterizes a hyperelliptic curve $C$ carrying the hyperelliptic involution $\iota$.
The hyperelliptic mapping class group $\UU_{g,n}$ then identifies with the fundamental group based at $[C]$ of the substack
$\cH_{g,n}$ of $\cM_{g,n}$ which parameterizes hyperelliptic curves. 
Observe that, in genus $1$ and $2$, all smooth projective curves are hyperelliptic, i.e.\ admit a finite morphism of degree $2$ onto $\PP^1$. 

As above, we define a \emph{level} of $\UU_{g,n}$ as a finite index subgroup $H<\UU_{g,n}$.
A characteristic finite index subgroup $K$ of $\Pi_{g,n}$ determines the \emph{geometric level}
$\UU_K:=\UU_{g,n}\cap\GG_K$, also defined as the kernel of the induced representation:
\[\rho_K\co \UU_{g,n}\to\out(\Pi_{g,n}/K).\]
In this way, we also get the \emph{abelian level} $\UU(m):=\UU_{g,n}\cap\GG(m)$ of order $m$, for $m\geq 2$.

The pro-$\cC$ topology on the fundamental group $\Pi_{g,n}$ determines the \emph{$\cC$-congruence topology on $\UU_{g,n}$} with fundamental 
system of neighborhoods of the identity $\{\UU_K\}_{K\lhd_\cC\Pi_{g,n}}$. A subgroup of $\UU_{g,n}$ is 
\emph{$\cC$-congruence open} if it contains some level $\UU_K$ for $K\lhd_\cC\Pi_{g,n}$. Let $\cUU_{g,n}^\cC$ be the corresponding
profinite completion of $\UU_{g,n}$ which we call the \emph{$\cC$-congruence completion of $\UU_{g,n}$} or also the \emph{pro-$\cC$ 
congruence hyperelliptic mapping class group}. 

The kernel of the natural epimorphism $\GG_{g,n+1}\to\GG_{g,n}$, induced filling in the last puncture, identifies with the fundamental group $\Pi_{g,n}$. 
The \emph{$\cC$-congruence completion} $\ccP_{g,n}^\cC$ of $\Pi_{g,n}$ is then the
completion with respect to the profinite topology induced on $\Pi_{g,n}$ by the $\cC$-congruence topology of $\GG_{g,n+1}$.
We will see (cf.\ Proposition~\ref{epi}) that there is a natural epimorphism $\ccP_{g,n}^\cC\to\hP_{g,n}^\cC$, which is an isomorphism for $\cC=\cF$ 
or $(p)$. In general, however, this map is not an isomorphism (e.g.\ for $\cC=\cS$) and it is a difficult and subtle question to determine when this happens. 
In Section~\ref{congruence}, we will prove the following theorem:

\begin{theorem}\label{main}\begin{enumerate}Let $\cC$ be a class of finite groups such that $\Z/2\in\cC$ and assume that $2g-2+n>0$ and $g\geq 1$:
\item For $n\geq 1$, the weak $\cC$-congruence property holds for $\UU_{g,n}$.
More precisely, there is a natural epimorphism of profinite groups $\cUU(2)^\cC\tura\hUU(2)^\cC$.
\item If $\cC$ is such that $\ccP_{g,k}^\cC\equiv\hP_{g,k}^\cC$, for all $k\geq 0$ (e.g.\ $\cC=\cF$ or $(2)$),
then the profinite group $\cUU_{g,n}^\cC$ is virtually the pro-$\cC$ completion of $\UU_{g,n}$. More precisely, there is
a natural isomorphism $\cUU(2)^\cC\equiv\hUU(2)^\cC$.
\end{enumerate}
\end{theorem}

A few comments are in order. As a particular case of (ii) of Theorem~\ref{main}, we obtain the subgroup congruence property for hyperelliptic
mapping class groups in the closed surface case which was left open in \cite{hyp}. 

For $\cC$ the class of $2$-groups and $g=1$, (ii) of Theorem~\ref{main} has already been proved in \cite{H-I} 
(for $\cC=\cF$ and $g=1$, this is instead a classical result by Asada \cite{Asada}). In the same paper, the authors prove that the pro-$p$ congruence 
subgroup property fails for $\GG_{1,1}$ if $p\geq 11$. We expect that the same is true in genus $>1$, so that the hypothesis $\Z/2\in\cC$
is probably not just a consequence of the proof. We instead believe that the hypothesis $n\geq 1$ in the first item of the theorem is 
a consequence of the proof. 

We ignore whether the standard version of the subgroup congruence property holds for classes $\cC$ of finite groups other than
$\cF$ or $(2)$. Ultimately, this issue is related to the fact that the pro-$\cC$ completion functor is right exact but, in general, not 
left exact (cf.\ Remark~\ref{obstruction}).

Even though the (weak) $\cC$-congruence property for mapping class groups was the original motivation behind this paper, at a more technical level, its 
core lies in the study of the combinatorial properties of pro-$\cC$ congruence mapping class groups developed in Section~\ref{modular section}
and~\ref{c&c}. The upshot is a description of centralizers of multitwist in the pro-$\cC$ congruence mapping class groups.
This then yields a description of their centers and the Birman exact sequences~(\ref{Birman2}) and~(\ref{Birman4}) which are central for the proof of 
Theorem~\ref{main}. These generalize results originally obtained by Matsumoto, Hoshi and Mochizuki (cf.\ Remark~\ref{Hoshi}). 

In Section~\ref{orbits}, we prove a generalized surface group version of a conjecture of Gelander and Lubotzky, proved for free groups of finite rank 
in \cite{PP}. Even though these results are essentially independent from the rest of the paper,
they establish interesting connections with Grothendieck-Teichm\"uller theory and low-dimensional topology.
\medskip

\noindent
{\bf Acnowledgements.} I thank Eduard Looijenga and Pavel Zalesskii for some useful conversations and suggestions.

\section{$\cC$-congruence topologies on mapping class groups}\label{geometric pro}
\emph{A $\cC$-congruence topology} on the mapping class group $\GG(S,x_1,\ldots,x_n)$, for $S$ a connected oriented surface of finite type and
a given a set $\{x_1,\ldots,x_n\}$ of distinct points on $S$, is the topology that it acquires as an (outer) automorphism group of some fundamental 
group, which we can associate to the datum of the pointed surface $(S,x_1,\ldots,x_n)$, endowed with its $\cC$-topology. As we will see below, there 
are several natural choices which lead to a priori distinct $\cC$-congruence topologies. However, as we will see in Sections~\ref{shortexact} 
and~\ref{compconf}, for some of the most interesting cases (e.g.\ for $\cC=\cF$ or $(p)$) all of them determine the same topology on 
the mapping class group.

\subsection{The standard $\cC$-congruence topologies}\label{standard}
For $x\in S$ a base point, let $\Pi:=\pi_1(S,x)$, $S_x:=S\ssm\{ x\}$ and $\Pi_x:=\pi_1(S_x,y)$, for $y\in S_x$. 
Then,  $\aut(\hP^{\cC})$ and $\out(\hP^{\cC})$ are profinite groups 
(cf.\ Theorem~1.1 in \cite{NS}) and the natural representations $\GG(S)\to \out(\Pi^{\cC})$ and $\GG(S_x)\cong \GG (S,x)\to \aut(\Pi^{\cC})$ 
are faithful. In fact, since $\cC$ contains $(p)$ for some prime $p$, it suffices to prove the claim for $\cC=(p)$, which, in this case,
follows from the $p$-separability and conjugacy $p$-separability of surface groups (cf.\ \cite{Paris} or also Theorem~A.1 and Theorem~4.3 in \cite{scc}).

\begin{definition}\label{S-congruence} 
The \emph{$\cC$-congruence topology} of $\GG (S)$ (resp.\  $\GG (S,x)$) is the topology inherited from the embedding $\GG (S)\hookra \out(\hP^{\cC})$ 
(resp.\  $\GG (S,x)\hookra \aut(\hP^{\cC})$) and its  \emph{$\cC$-congruence completion}, denoted $\cGG(S)^{\cC}$ (resp.\ $\cGG(S,x)^{\cC}$),
is the associated completion or, equivalently, the closure of the image of this embedding. We also refer to these groups as to the 
\emph{pro-$\cC$-congruence mapping class groups $\cGG(S)^{\cC}$ and $\cGG(S,x)^{\cC}$}.
\end{definition}

A basic and very subtle question about the $\cC$-congruence topologies introduced above is whether the natural isomorphism of mapping class groups
$\GG(S_x)\cong\GG(S,x)$ induces an isomorphism $\cGG(S_x)^\cC\cong\cGG(S,x)^\cC$ of the respective $\cC$-congruence completions.
A related issue, which will turn out to be actually equivalent, is whether the two $\cC$-congruence topologies induce the same topology on
the fundamental group $\Pi$ considered as a subgroup of $\GG(S_x)$ and $\GG(S,x)$ respectively. 

Since the center of a pro-$\cC$ surface group is trivial (cf.\ Lemma~\ref{central}), the natural homomorphism $\inn\co\Pi\to\GG(S,x)\subset\aut(\hP^\cC)$
induces on $\Pi$ the pro-$\cC$ topology and the natural short exact sequence $1\to\hP^{\cC}\to\aut(\hP^{\cC})\to\out(\hP^{\cC})\to 1$ determines a Birman
short exact sequence of profinite groups:
\begin{equation}\label{Birman1}
1\to\hP^{\cC}\to\cGG (S, x)^{\cC}\to\cGG(S)^{\cC}\to 1.
\end{equation}
However, it is not clear whether the series of homomorphisms $\Pi\to\GG(S_x)\subset\out(\hP_x^\cC)$ induces the pro-$\cC$ topology on $\Pi$.
In this case, an element $\gm\in\Pi$ whose free homotopy class contains a simple closed curve is sent to the bounding pair map on $S_x$ associated
to a tubular neighborhood of a simple representative of $\gm$. We call the topology induced by the associated embedding $\Pi\subset\out(\hP_x^\cC)$ 
the \emph{$\cC$-congruence topology} on $\Pi$ and denote by $\ccP^\cC$ the corresponding profinite completion. In Proposition~\ref{epi}, we will 
show that, for $\cC=\cF$ or $(p)$, we have $\ccP^\cC\equiv\hP^\cC$ and, in Corollary~\ref{faithful}, we will show that $\cGG(S_x)^\cC\cong\cGG(S,x)^\cC$,
if and only if, $\ccP^\cC\equiv\hP^\cC$.

\subsection{The tangential base point representation}\label{tangential}
From the natural isomorphism of mapping class groups $\GG(S_x)\cong\GG(S,x)$, we deduce 
that there is a (faithful) representation $\GG(S_x)\to\aut(\hP^\cC)$. However, it is not clear whether this representation factors through 
the natural embedding $\GG(S_x)\subset\cGG(S_x)^\cC$ and then a representation $\cGG(S_x)^\cC\to\aut(\hP^\cC)$ with image $\cGG(S,x)^\cC$. 
In this section, we will show that this is the case.

Let us fix a tangential base point at $x\in S$ for the surface $S_x=S\ssm\{ x\}$. To be precise, a base point  $x\in S$ and a ray in the 
tangent space $T_xS$.  Let $q\co S_\circ\to  S$ be the real oriented  blow-up of $S$ at $x$.  Then, 
$S_\circ$  is a surface with boundary whose interior is mapped  diffeomorphically onto $S_x$ and whose 
boundary $\partial S_\circ$ is the preimage of $x$. The points of $\partial S_\circ$ are in bijective 
correspondence with the rays in the tangent space $T_xS$ so that the given ray defines a base point  $x_\circ\in \partial S_\circ$.  
We put $\Pi_\circ:=\pi_1(S_\circ,x_\circ)$ so that $\GG(S_\circ)$ is a 
subgroup of  the outer automorphism group of $\Pi_\circ$ and we let $\cGG(S_\circ)^\cC$ to be the profinite completion induced
by the monomorphism $\GG (S_\circ)\hookra\out(\hP_\circ^\cC)$. The natural isomorphism $\GG(S_\circ)\cong\GG(S_x)$ then
induces an isomorphism of the respective $\cC$-congruence completions $\cGG(S_\circ)^\cC\cong\cGG(S_x)^\cC$.

Let $u\in\Pi_\circ$ be an element whose free isotopy class contains $\partial S_\circ$, let $\cU:=\{u\}^{\hP_\circ^\cC}$ be its conjugacy class in 
$\hP_\circ^\cC$, let then $\aut_u(\hP^\cC_\circ)$ and $\out_{\cU}(\hP^\cC_\circ)$ be the stabilizers of $u$ and $\cU$ in the groups 
$\aut(\hP^\cC_\circ)$ and $\out(\hP^\cC_\circ)$, respectively. Observe that $\inn u$ lies in the center of the group $\aut_u(\hP^\cC_\circ)$.

\begin{lemma}\label{tangential1}Let us denote by $\Inn(u^{\ZZ^\cC})$ the procyclic subgroup of $\aut(\hP^\cC_\circ)$ generated by $\inn u$. 
Then the natural homomorphism $\hat{\Theta}^\cC_u\co\left.\aut_u(\hP^\cC_\circ)\right/\Inn(u^{\ZZ^\cC})\to\out_{\cU}(\hP^\cC_\circ)$ is an 
isomorphism of profinite groups.
\end{lemma}

\begin{proof}Let us construct an inverse $\hat{\Kappa}^\cC_u$ for the homomorphism $\hat{\Theta}^\cC_u$. 
For this, we need the following lemma from \cite{BZ} (cf.\ Proposition~3.5 ibid.):

\begin{lemma}\label{central}Let $\cC$ be an extension closed class of finite groups.
Let $\hP^\cC$ be the pro-$\cC$ completion of a non-abelian surface group $\Pi$. 
Let $x\neq 1$ be a primitive element of $\Pi$ considered as an element of $\hP^\cC$. 
Then, for all $n\in\Z\ssm\{0\}$, there holds:
\[N_{\hP^\cC}(x^{n\ZZ^\cC})=N_{\hP^\cC}(x^{\ZZ^\cC})=x^{\ZZ^\cC}.\]
\end{lemma}

From Lemma~\ref{central}, it follows 
that for $f\in\out_{\cU}(\hP^\cC_\circ)$, modulo $\Inn(u^{\ZZ^\cC})$, there is a unique lift $\td{f}\in\aut_u(\hP^\cC_\circ)$.
We then let $\Kappa_u(f):=[\td{f}]$, where $[\td{f}]$ denotes the coset of $\td{f}$ in the quotient $\left.\aut_u(\hP^\cC_\circ)\right/\Inn(u^{\ZZ^\cC})$.
It is clear that $\Kappa_u\circ\Theta_u=\Id$ and $\Theta_u\circ\Kappa_u=\Id$, which completes the proof of Lemma~\ref{tangential1}.
\end{proof}

The natural epimorphism $\hP^\cC_\circ\to\hP^\cC$ induces a natural continuous homomorphism 
\[\hat{\Omega}^\cC_u\co\left.\aut_u(\hP^\cC_\circ)\right/\Inn(u^{\ZZ^\cC})\to\aut(\hP^\cC).\]
Let us then observe that, even though the isomorphism $\hat{\Kappa}^\cC_u$ and the homomorphism $\hat{\Omega}^\cC_u$ depend on the choice of 
$u\in\cU$, their composition $\hat{\Omega}^\cC_u\circ\hat{\Kappa}^\cC_u$ only depends on the conjugacy class $\cU$, which is determined by the 
point $x\in S$. Therefore, restricting $\hat{\Omega}^\cC_u\circ\hat{\Kappa}^\cC_u$ to $\cGG(S_x)^\cC\equiv\cGG(S_\circ)^\cC$, we obtain a natural
continuous epimorphism
\[\hat{\Delta}^\cC_x\co\cGG(S_x)^\cC\tura \cGG (S,x)^\cC,\] 
which only depends on the point $x\in S$. Hence, we obtain a natural representation $\cGG(S_x)^\cC\to\aut(\hP^\cC)$ which extends the
representation $\GG(S_x)\hookra\aut(\hP^\cC)$ considered at the beginning of this section. 

The kernel of $\hat{\Delta}^\cC_x$ clearly measures the difference between the $\cC$-congruence topologies on the mapping class group 
$\GG(S_x)$, associated to the representations $\GG(S_x)\hookra\out(\hp_1(S_x)^\cC)$ and $\GG(S_x)\cong\GG(S,x)\hookra\aut(\hP^\cC)$.

The epimorphism $\hat{\Delta}^\cC_x$ can be recovered from the inner action of $\cGG(S_x)^\cC$ on the $\cC$-congruence
completion $\check{\Pi}^\cC$ of $\Pi$. 
In Section~\ref{shortexact}, we will actually show that $\ker\hat{\Delta}^\cC_x=\{1\}$, if and only if, $\check{\Pi}^\cC\equiv\hP^\cC$.
For the moment, let us observe the following:

\begin{proposition}\label{epi}Let $\check{\Pi}^\cC$ be the $\cC$-congruence completion of $\Pi$ defined in Section~\ref{standard}:
\begin{enumerate}
\item There is a natural epimorphism $\check{\Pi}^\cC\tura\hP^\cC$.
\item For $\cC=\cF$ or $(p)$, we have $\check{\Pi}^\cC\equiv\hP^\cC$.  
\end{enumerate}
\end{proposition}

\begin{proof}The restriction to $\check{\Pi}^\cC$ of the epimorphism 
$\hat{\Delta}^\cC_x\co\cGG(S_x)^\cC\tura \cGG (S,x)^\cC$ induces an epimorphism $\check{\Pi}^\cC\tura\hP^\cC$,
which is obviously an isomorphism for $\cC=\cF$. For $\cC=(p)$, this instead follows from Lemma~4.5.5 in \cite{RZ}.
\end{proof}

\subsection{The configuration space $\cC$-congruence topology}\label{conftop}
For a given a set $\{x_1,\ldots,x_n\}$ of distinct points on $S$ and $n\geq 2$, there is another natural representation to consider.
Let us denote by $S(n)$ the configuration space of $n$ points on $S$, by $\ol y=:\{y_1,\ldots,y_n\}$ 
a point of $S(n)$ and by $\Pi(n)$ its fundamental group based at $\ol{x}=:\{x_1,\ldots,x_n\}\in S(n)$. 
Let $\GG(S,\ol{x})$ be the mapping class group of the $n$-pointed surface $(S,x_1,\ldots,x_n)$. 
A self-homeomorphism of the pointed surface $(S,x_1,\ldots,x_n)$ induces one of the pointed space $(S(n),\ol{x})$.
There is then a natural faithful representation:
\[\rho_n\co\GG(S,\ol{x})\hookra\aut(\Pi(n)).\]
Alternatively, the representation $\rho_n$ is defined by restricting the inner automorphisms of $\GG(S,\ol{x})$ to its normal subgroup $\Pi(n)$.
Let $\hP(n)^\cC$ be the pro-$\cC$ completion of $\Pi(n)$, then $\rho_n$ induces a representation:
\[\rho_n^\cC\co\GG(S,\ol{x})\to\aut(\hP(n)^\cC).\]
We define the \emph{$\cC$-congruence completion} of $\GG(S,\ol{x})$ (or also the \emph{pro-$\cC$-congruence mapping class group}), denoted 
$\cGG(S,\ol{x})^{\cC}$, to be the closure of the image of the representation $\rho_n^\cC$. In Section~\ref{compcong}, we will compare this congruence
topology with those introduced in Section~\ref{standard}.

\subsection{The $\cS$-congruence topology is not virtually pro-$\cS$}\label{solvable}
In this section, we will show that, for $\cC=\cS$, the class of finite solvable groups, the pro-$\cS$ congruence mapping class group $\cGG_{g,n}^\cS$ is
not virtually pro-$\cS$:

\begin{proposition}\label{novirtual}For $g\geq 1$, the pro-$\cS$ congruence mapping class group $\cGG_{g,n}^\cS$ is not virtually pro-$\cS$. 
In particular, $\out(\hP_{g,n}^\cS)$ and $\aut(\hP_{g,n}^\cS)$ are not virtually pro-$\cS$ either.\end{proposition}

\begin{proof}Let $\hP_{g,n}^\mathrm{nil}$ be the profinite nilpotent completion of $\Pi_{g,n}$. There is a natural epimorphism with characteristic kernel 
$\hP_{g,n}^\cS\tura\hP_{g,n}^\mathrm{nil}$ and then a natural representation $\cGG_{g,n}^\cS\to\out(\hP_{g,n}^\mathrm{nil})$. Let us denote by
$\cGG_{g,n}^\mathrm{nil}$ the image of this representation. The decomposition 
$\aut(\hP_{g,n}^\mathrm{nil})\cong\prod_{p\,\,\mathrm{prime}}\aut(\hP_{g,n}^{(p)})$ induces one 
$\cGG_{g,n}^\mathrm{nil}\cong\prod_{p\,\,\mathrm{prime}}\cGG_{g,n}^{(p)}$. For every prime $p>0$, there is an epimorphism
$\cGG_{g,n}^{(p)}\tura\Sp_{2g}(\Z/p)$ and every open subgroup $U$ of $\cGG_{g,n}^\mathrm{nil}$ contains infinitely many factors $\cGG_{g,n}^{(p)}$. 
Therefore, the group $\cGG_{g,n}^\mathrm{nil}$ is not virtually pro-$\cS$ and so the same is true for $\cGG_{g,n}^\cS$.
\end{proof}

\begin{remark}\label{novirtual3}By the methods of \cite{Asada}, it is not difficult to show that the abelian level $\cGG_{1,1}^\cS(4)$ identifies with an 
open subgroup of $\cGG_{0,4}^\cS$ (see also the short exact sequence~(\ref{abelian_2_pro})). Therefore, Proposition~\ref{novirtual} also holds in genus 
$0$. We leave the details to the interested reader.
\end{remark}

\section[Modular subgroups]{Modular subgroups of pro-$\cC$-congruence mapping class groups}
\label{modular section}
A fundamental feature of the mapping class group $\GG(S)$ of a connected  oriented surface of finite negative Euler characteristic $S$ is that
the stabilizer of a set of isotopy classes of disjoint simple closed curves (embedded circles) on $S$  is described in terms of mapping class groups
of the subsurfaces obtained cutting $S$ along those curves. In this section, we will show that a similar property holds for the pro-$\cC$ congruence 
mapping class group.

\subsection{The closure of stabilizers}\label{the closure}
The following theorem is a generalization to the pro-$\cC$ case of Lemma~4.6 of \cite{faith}.
Let us recall that a simple closed curve (an embedded circle) on $S$
is \emph{peripheral} if it bounds on one side a disc with a single puncture.
We denote by $\GG(S\ssm\gm)$ the direct product of the pure mapping class groups of the connected components of the surface $S\ssm\gm$:

\begin{theorem}\label{stabilizer} Let $S$ be a connected  oriented surface of finite negative Euler 
characteristic and let $\gm\subset S$ be  a non-peripheral simple closed curve on $S$.
Denote by  $\GG(S)_{\vec{\gm}}$ the subgroup of elements of $\GG(S)$ which preserve the free homotopy
class of $\gm$ and a fixed orientation on it and by $\cGG (S)_{\vec{\gm}}^{\cC}$ the closure of this 
subgroup in the $\cC$-congruence completion $\cGG(S)^{\cC}$ of $\GG (S)$. Then there is 
a natural surjection $\cGG(S)_{\vec{\gm}}^{\cC}\twoheadrightarrow\cGG(S\ssm\gm)^{\cC}$ whose 
kernel is topologically generated by the Dehn twist $\tau_\gamma$ (making it isomorphic to $\ZZ^\cC$). 
\end{theorem}

We will need the following lemma from \cite{BZ} (cf.\ Theorem~3.6 ibid.):

\begin{lemma}\label{subsurface}
Let $\sg$ be a closed one-dimensional submanifold of an oriented surface $S$, $S_0$
a connected component of $S\ssm\sg$ and $x\in S_0$. Then the injection 
$\pi_1(S_0,x)\hookrightarrow \pi_1(S,x)$ induces an  injection $\hp_1(S_0,x)^\cC\hookrightarrow \hp_1(S,x)^\cC$
with closed image.  Moreover, that image is its own normalizer in $\hp_1(S,x)^\cC$.
\end{lemma}

\begin{proof}[Proof of Theorem \ref{stabilizer}]
Every mapping class which fixes the homotopy class of $\gm$ fixes its isotopy class and hence contains a member which
leaves $\gm$  invariant in an orientation preserving manner. It is easily checked that this identifies $\GG(S)_{\vec{\gm}}$ 
with the connected component group of diffeomorphisms of $S$ having that property and so we have a natural homomorphism 
$\GG(S)_{\vec{\gm}}\to \GG(S\ssm\gm)$. This homomorphism is surjective, for every
mapping class of $S\ssm\gm$ is representable by a  diffeomorphism that is the identity on the trace of  
a neighborhood of $\gm$ on $S\ssm\gm$. It is also easy to see that the kernel is generated by $\tau_\gm$. 

Denote by $q\co S^\gm\to S$ the real oriented blowup of $S$ along $\gm$. This is a surface with boundary $\partial S^\gm$. 
Its interior maps isomorphically onto $S\ssm\gm$ and the preimage of $\gm$ consists of two boundary components $\gm_\pm$, 
each mapping isomorphically onto $\gm$ and corresponding to the choice of an orientation of $\gm$. 
We choose a base point $x\in \gm$ and denote by $x_\pm$ its preimage in $\gm_\pm$. 

In what follows we assume that $S\ssm\gm$ is connected, the disconnected case being somewhat easier to treat. We abbreviate 
$\Pi:=\pi_1(S,x)$ and $\Pi_\gm:=\pi_1(S^\gm, x_-)$. Then $\Pi$ can then be obtained from  $\Pi_\gm$ as follows: choose
an embedded interval $\delta$ from $x_-$ to $x_+$ so that its image $q(\delta)$ in $S$ is an embedded circle and denote by 
$c_-, c_+ \in \Pi_\gm$ the elements defined by respectively $\gm_-$ and the loop which first  traverses $\delta$, then
$\gm_+$ and then returns via the inverse of $\delta$.  Then $\Pi$ can be identified with the quotient of the free product of $\Pi_\gm$ and 
a multiplicatively written copy $d^\Z$ of $\Z$ by the relation $d c_-d^{-1}=c_+$, with $d$ mapping to the class of $q(\delta)$ in $\Pi$. 
This is an instance of the  HNN construction; in particular, the homomorphism $\Pi_\gm\to \Pi$ is injective. 
According to Lemma \ref{subsurface}, this induces an injection $\hat\Pi_\gm^\cC\to \hat\Pi^\cC$ with closed image and with the property 
that this image is its own normalizer. We thus obtain a homomorphism $\GG(S)_{\vec{\gm}}\to \out (\Pi_\gm)$ which enjoys the property 
that it is continuous for the $\cC$ topology induced by $\Pi$ resp.\ $\Pi_\gm$. This induces a homomorphism 
$\cGG(S)_{\vec{\gm}}\to\out (\Pi_\gm^\cC)$ with image $\cGG(S^\gm)\cong \cGG(S\ssm\gm)$.

Now let $f$ be in the kernel of the surjection $\cGG(S)_{\vec\gm}\to \cGG(S^\gm)$.
By definition $f$ lifts to an automorphism $\tilde{f}$ of $\hat\Pi^{\cC}$ that is the identity on $\hP_\gm^{\cC}$. In particular, 
$\tilde{f}$ fixes $c_\pm$ and so
\[
\tilde{f}(d)c_-\tilde{f}(d)^{-1}=\tilde{f}(dc_-d^{-1})=\tilde{f}(c_+)=c_+=
d c_-d^{-1}. 
\]
Hence $d^{-1}\tilde{f}(d)$ centralizes $c_-$.
By Proposition~3.5 in \cite{BZ}, this implies that $\tilde{f}(d)=dc_-^k$, for some $k\in\hat{\Z}^\cC$. 
In terms of the HNN decomposition above, the profinite Dehn twist $\tau_\gm^{-k}$ takes $d$ to $dc_-^k$ and induces 
the identity in $\hP_\gm^{\cC}$. Since  $\hat\Pi^\cC_\gm$  and $d$ generate a dense subgroup of $\hat\Pi^\cC$,  
it follows that $f=\tau_\gm^{-k}$.
\end{proof}

\subsection{The complex of pro-$\cC$ curves on $S$}\label{the complex}
Sets of isotopy classes of disjoint simple closed curves on the surface $S$ can be arranged so that they form a simplicial complex: 

\begin{definition}\label{curve-complex}A \emph{multicurve} $\sg$ on $S$ is a set of disjoint
simple closed curves on $S$ such that $S\ssm\sg$ does not contain discs or one-punctured discs.
The complex of curves $C(S)$ is the abstract simplicial complex whose simplices are isotopy classes of multicurves on $S$.
\end{definition}

For $S=S_{g,n}$, it is easy to check that the combinatorial dimension of $C(S_{g,n})$ is
$n-4$ for $g=0$ and $3g-4+n$ for $g\geq 1$. There is a natural simplicial action of $\GG(S)$ on $C(S)$.
In this subsection, we will define a pro-$\cC$ version of the complex of curves $C(S)$, generalizing definitions and results 
given in Section~3 and 4 of \cite{faith}.

Let $\cL=C(S)_0$ be the set of isotopy classes of nonperipheral simple closed curves on $S$. Let $\Pi/\!\sim$ be the set of conjugacy
classes of elements of $\Pi$ (:$=\pi_1(S,x)$) and let $\cP_2(\Pi/\!\sim)$ be the set of unordered pairs of elements of $\Pi/\!\sim$.
For a given $\gm\in\Pi$, let us denote by $\gm^{\pm 1}$ the set $\{\gm,\gm^{-1}\}$
and by $[\gm^{\pm 1}]$ its equivalence class in $\cP_2(\Pi/\!\sim)$. Let us then define the
natural embedding $\iota\co\cL\hookra\cP_2(\Pi/\!\sim)$, choosing, for
an element $\gm\in\cL$, an element $\vec{\gm}_\ast\in\Pi$ whose free homotopy class
contains $\gm$ and letting $\iota(\gm):=[\vec{\gm}_\ast^{\pm 1}]$.

Let $\hP^\cC/\!\sim$ be the set of conjugacy classes of elements of $\hP^\cC$ and $\cP_2(\hP^\cC/\!\sim)$ the profinite set of
unordered pairs of elements of $\hP^\cC/\!\sim$. Since $\Pi$ is conjugacy $p$-separable (cf.\ \cite{Paris} or Theorem~4.3 \cite{scc}),  
the set $\Pi/\!\sim$ embeds in the profinite set $\hP^\cC/\!\sim$. So, let us define the set of \emph{nonperipheral pro-$\cC$ simple closed curves 
(or circles)} $\hL^\cC$ on $S$ to be the closure of the set $\iota(\cL)$ inside the profinite set $\cP_2(\hP^\cC/\!\sim)$. 
An ordering of the set $\{\alpha,\alpha^{-1}\}$ is preserved by the conjugacy action and
defines \emph{an orientation} for the associated equivalence class $[\alpha^{\pm1}]\in\hL^\cC$.

For all $k\geq 0$, there is a natural embedding of the set $C(S)_{k-1}$ of isotopy classes of
multicurves on $S$ of cardinality $k$ into the profinite set $\cP_k(\hL^\cC)$
of unordered subsets of $k$ elements of $\hL^\cC$. Let us then define the set of \emph{pro-$\cC$ multicurves}
on $S$ as the union of the closures of the sets $C(S)_{k-1}$ inside the profinite sets $\cP_k(\hL^\cC)$, for all $k>0$.

A \emph{simplicial profinite complex} is an abstract simplicial complex whose 
set of vertices is endowed with a profinite topology such that the sets of $k$-simplices, 
with the induced topologies, are compact and then profinite, for all $k\geq 0$. For these simplicial 
complexes, the procedure which associates to an abstract simplicial complex and an ordering of
its vertex set a simplicial set produces a simplicial profinite set.

\begin{definition}\label{geopro}Let $L(\hP^\cC)$ be the abstract simplicial profinite complex whose simplices are pro-$\cC$ multicurves on
$S$. The abstract simplicial profinite complex $L(\hP^\cC)$ is called the \emph{complex of pro-$\cC$ curves on $S$}.
\end{definition}

There is a natural continuous action of the pro-$\cC$ congruence mapping class group $\cGG(S)^\cC$ on the complex of
pro-$\cC$ curves $L(\hP^\cC)$. There are finitely many orbits of $\cGG(S)^\cC$ in $L(\hP^\cC)_k$ each containing an element of
$C(S)_k$, for $k\geq 0$. It is not difficult to show (cf.\ Section~4 \cite{faith}) that these orbits correspond to the possible topological
types of the surface $S\ssm\sg$, for $\sg$ a multicurve on $S$.

Let $G$ be a discrete group acting with a finite number of orbits on a set $X$ and let $\td{G}$ be 
the profinite completion of $G$ with respect to an inverse system $\{U_\ld\}_{\ld\in\Ld}$ of finite
index normal subgroups. The \emph{$\td{G}$-completion $\td{X}$ of the set $X$} is defined to be
the inverse limit $\varprojlim_{\ld\in\Ld}X/U_\ld$. The profinite set $\td{X}$ is naturally endowed 
with a continuous $\td{G}$-action and a $G$-equivariant map $X\ra\td{X}$. It is characterized by 
the following universal property. Given a profinite set $Y$ endowed with a continuous $\td{G}$-action 
and a $G$-equivariant map $f\co X\ra Y$, there exists a unique continuous 
$\td{G}$-equivariant map $\td{f}\co\td{X}\ra Y$ extending $f$.

Let us also observe that, if $\td{U}$ is an open subgroup of $\td{G}$ and $U$ is its inverse image
in $G$ via the canonical map $G\ra\td{G}$, then a $G$-set $X$ is also a $U$-set and its 
$\td{U}$-completion is naturally isomorphic to its $\td{G}$-completion.

The main result of \S 4 in \cite{faith} (cf.\ Theorem~4.2 in \cite{faith}) is that, for $\cC$ the class of all finite
groups, the profinite set $\hL(S)$ is the $\cGG(S)$-completion of the discrete $\GG(S)$-set $\cL(S)$,
i.e. there is a natural continuous isomorphism of $\cGG(S)$-sets:
\[\hL(S)\cong\ilim_{\ld\in\Ld}\,\cL(S)\left/\GG_\ld\right.,\]
where $\{\GG_\ld\}_{\ld\in\Ld}$ is a tower of finite index normal subgroup of $\GG(S)$ 
which forms a fundamental system of neighbourhoods of the identity for the full congruence topology.
In fact, this result holds for any extension closed class of finite groups $\cC$:

\begin{theorem}\label{isomorphism}Let $\{\GG_\ld\}_{\ld\in\Ld}$ be a tower of finite index normal subgroup of $\GG(S)$
which forms a fundamental system of neighbourhoods of the identity for the $\cC$-congruence topology. There is then a natural 
continuous isomorphism of $\cGG(S)^\cC$-sets:
\begin{equation}\label{realization}\hL^\cC(S)\cong\ilim_{\ld\in\Ld}\,\cL(S)\left/\GG_\ld\right..\end{equation}
\end{theorem}

\begin{proof}In order to prove Theorem~\ref{isomorphism}, we just need to rephrase the proof of Theorem~4.2 in \cite{faith}
in terms of pro-$\cC$ completions and $\cC$-congruence completions instead of full profinite completions and full congruence completions 
of all the objects involved, where Lemma~4.3 in \cite{faith} is replaced by its pro-$\cC$ version (cf.\ Lemma~\ref{central} in this paper).
\end{proof}

Given a simplex $\sg\in C(S)$, let us denote also by $\sg$ its image in $L(\hP^\cC)$ and by $\cGG(S)^\cC_\sg$ the corresponding 
$\cGG(S)^\cC$-stabilizer. From Proposition~6.5 in \cite{PFT} and Theorem~\ref{isomorphism}, it follows that $\cGG(S)^\cC_\sg$
is the closure of the stabilizer $\GG(S)_\sg$ (for the action of $\GG(S)$ on $C(S)$) in $\cGG(S)^\cC$.

Since every $\sg\in L(\hP^\cC)$ is in the $\cGG(S)^\cC$-orbit of some simplex in the image of $C(S)$, this gives a complete description
of stabilizers for the action of $\cGG(S)^\cC$ on the pro-$\cC$ curve complex $L(\hP^\cC)$. The precise result is the following, where, for 
a simplex $\sg\in C(S)$, we denote by $\GG(S\ssm\sg)$ the direct product of the pure mapping class groups of the connected components of 
the surface $S\ssm\sg$:

\begin{theorem}\label{stab pro-curves}Let $\sg\in L(\hP^\cC)$ be a simplex in the image of $C(S)$. Let $\Sigma_{\sg^{\pm}}$ be the group of 
permutations on the set $\sg^{\pm}:=\sg^{+}\cup\sg^{-}$ of oriented simple closed curves in $\sg$.
Then, the stabilizer $\cGG(S)^\cC_\sg$ of $\sg$ for the action of $\cGG(S)^\cC$ on $L(\hP^\cC)$ fits into the two exact sequences:
\[\begin{array}{c}
1\ra\cGG(S)^\cC_{\vec{\sg}}\to\cGG(S)^\cC_\sg\to\Sigma_{\sg^{\pm}},\\
1\ra\prod\limits_{i=0}^k\tau_{\gm_i}^{\ZZ^\cC}\to\cGG(S)^\cC_{\vec{\sg}}\to\cGG(S\ssm\sg)^\cC\to 1.
\end{array}\]
\end{theorem}

\begin{proof}Exactness of the upper sequence is obvious. Right exactness and exactness in the middle of the one below both follow from 
Theorem~\ref{stabilizer} and induction. Left exactness is equivalent to the claim that the levels $\GG^{K_\ell,(m)}$ of $\GG(S)$ introduced 
in \cite{sym} induce the pro-$\cC$ topology on the abelian subgroup $\prod_{i=0}^k\tau_{\gm_i}^\Z$ of $\GG(S)$. 
The claim then follows from the explicit description of the local monodromy representation for these level structures given in 
Proposition~3.11 \cite{sym}, where for $K$ we take a $\cC$-open subgroup of $\Pi$ satisfying the hypotheses of 
Lemma~ 3.10 \cite{sym} and for $\ell, m\geq 2$ integers such that $\Z/\ell$ and $\Z/m\in\cC$.
\end{proof}

\section[Centralizers and centers]{Centralizers of pro-$\cC$ Dehn twists and centers of pro-$\cC$ congruence mapping class groups}\label{c&c}
The stabilizer in the mapping class group $\GG(S)$ of a set of isotopy classes of disjoint simple closed curves on $S$ is the centralizer of the product of
the Dehn twists associated to those curves. We will show that this result holds also for the pro-$\cC$ congruence mapping class group $\cGG(S)^\cC$ 
but the proof is much more difficult and requires to develop entirely different tools.
This will allow to determine the centers of pro-$\cC$ congruence mapping class groups and then to provide a Birman exact sequence for them. 

\subsection{A linearization of the complex of pro-$\cC$ curves}\label{linearization}
In this section, we will extend the results of Section~5 in \cite{BZ} to the pro-$\cC$ case.
Let $K$ be an open characteristic subgroup of $\hP^\cC$ and let $p_K\co S_K\ra S$ be the
associated normal unramified covering of Riemann surfaces with covering transformation group
$G_K:=\hP^\cC/K$. Let then $\ol{S}_K$ be the closed Riemann surface obtained from $S_K$
filling in its punctures and, for a commutative unitary ring of coefficients $A$, let $H_1(\ol{S}_K,A)$
be its first homology group. There is then a natural map $\psi_K\co K\ra H_1(\ol{S}_K,A)$.

For a given $\gm\in\hL^\cC$, let us denote by the same letter an element of the profinite group
$\hP^\cC$ in the class of the given profinite simple closed curve and let $\nu_K(\gm)$ be the smallest positive
integer such that $\gm^{\nu_K(\gm)}\in K$. For a profinite multicurve $\sg\in L(\hP^\cC)$, let us
also denote by $\sg$ a subset of $\hP^\cC$ in the class of the given multicurve. We let then
$V_{K,\sg}$ be the primitive $A[G_K]$-submodule of $H_1(\ol{S}_K,A)$ generated by the $G_K$-orbit of
$\{\psi_K(\gm^{\nu_K(\gm)})\}_{\gm\in\sg}$.

Let $\mathrm{Gr}_{G_K}(H_1(\ol{S}_K,A))$ be the absolute Grassmanian of primitive
$A[G_K]$-submodules of the homology group $H_1(\ol{S}_K,A)$, that is to say the disjoint union of the
Grassmanians of primitive, $k$-dimensional, $A[G_K]$-submodules of
$H_1(\ol{S}_K,A)$, for all $1\leq k\leq\mathrm{rank}\,H_1(\ol{S}_K,A)$.

For $A_p$ equal to the ring of $p$-adic integers $\Z_p$ or the finite field
$\F_p$, the absolute Grassmanian $\mathrm{Gr}_{G_K}(H_1(\ol{S}_K,A_p))$ has a natural structure of
profinite space, while, for $A_p=\Q_p$, it is a locally compact totally disconnected Hausdorff space.
In all cases, for $\sg\in L(\hP^\cC)$, the assignment $\sg\mapsto V_{K,\sg}$ defines
a natural continuous $\cGG(S)^\cC$-equivariant map:
\[\Psi_{K,p}\co L(\hP^\cC)\to\mathrm{Gr}_{G_K}(H_1(\ol{S}_K,A_p)).\]

The following theorem is then a straightforward generalization of Theorem~5.6 in \cite{BZ}, where, in the proof, we just need to replace the word
"open" by "$\cC$-open":

\begin{theorem}\label{embedding}For $p>0$ a prime number, let $A_p=\F_p$, $\Z_p$ or $\Q_p$.
There is a natural continuous $\cGG(S)^\cC$-equivariant injective map:
\[\widehat{\Psi}_p:=\prod_{K\unlhd_\cC\hP^\cC}\Psi_{K,p}\co L(\hP^\cC)\hookra\prod_{K\unlhd_\cC\hP^\cC}\mathrm{Gr}_{G_K}(H_1(\ol{S}_K,A_p)).\]
\end{theorem}

\subsection{Centralizers of pro-$\cC$ multitwists}\label{centralizers}
The set $\cL$ of isotopy classes of nonperipheral simple closed curves on $S$ parametrizes the set of Dehn
twists of $\GG(S)$, which is the standard set of generators for this group. The assignment $\gm\mapsto\tau_\gm$, for $\gm\in\cL$, 
defines an embedding $d\co\cL\hookra\GG(S)$, thus relating, in an immediate way, the elementary topology of $S$ with the group structure of $\GG(S)$.

The set $\{\tau_\gm\}_{\gm\in\cL}$ of all Dehn twists of $\GG(S)$ is closed under conjugation and falls in a finite set
of conjugacy classes which are in bijective correspondence with the possible topological types of the Riemann surface
$S\ssm\gm$. The set of \emph{pro-$\cC$ Dehn twists} is defined
to be the closure of the image of the set $\{\tau_\gm\}_{\gm\in\cL}$ inside $\cGG(S)^\cC$. This is the same as the union of the conjugacy
classes in $\cGG(S)^\cC$ of the images of the Dehn twists of $\GG(S)$.

The composition of the embedding $d$ with the natural monomorphism $\GG(S)\hookra\cGG(S)^\cC$ is a $\GG(S)$-equivariant map 
$\cL\hookra\cGG(S)^\cC$, where we let $\GG(S)$ act by conjugation on $\cGG(S)^\cC$. From the universal property of the $\cGG(S)^\cC$-completion 
and Theorem~\ref{isomorphism}, it then follows that this map extends to a continuous $\cGG(S)^\cC$-equivariant map
$\hat{d}^\cC\co\hL^\cC\ra\cGG(S)^\cC$, whose image is the set of pro-$\cC$ Dehn twists. We then let $\tau_\gm:=\hat{d}^\cC(\gm)$, for all $\gm\in\cL^\cC$.
Similarly, to a multicurve $\sg=\{\gm_1,\ldots,\gm_s\}\in\cL^\cC$ and a multi-index $(h_1,\ldots,h_s)\in(\ZZ^\cC)^s$, we associate the 
\emph{pro-$\cC$ multitwist} $\tau_{\gm_1}^{h_1}\cdot\ldots\cdot\tau_{\gm_s}^{h_s}$. 

In this section, we want to extend to pro-$\cC$ congruence mapping class groups Theorem~6.1 of \cite{BZ} and then its Corollary~6.2, 
which describes the centralizer of a profinite multitwist. We will actually prove a stronger version of those results:

\begin{theorem}\label{multitwists}Let $\sg=\{\gm_1,\ldots,\gm_s\}$ and
$\sg'=\{\delta_1,\ldots,\delta_t\}$ be two pro-$\cC$ multicurves on the surface $S$.
Suppose that, for multi-indices $(h_1,\ldots,h_s)\in m_\sg\cdot(\Z\ssm\{0\})^s$ and $(k_1,\ldots,k_t)\in m_{\sg'}\cdot(\Z\ssm\{0\})^t$, 
where $m_\sg,m_{\sg'}\in(\ZZ^\cC)^\ast$, there is an identity:
\[\tau_{\gm_1}^{h_1}\cdot\ldots\cdot\tau_{\gm_s}^{h_s}=
\tau_{\delta_1}^{k_1}\cdot\ldots\cdot\tau_{\delta_t}^{k_t}\in\cGG(S)^\cC.\]
Then, we have:
\begin{enumerate}
\item $t=s$;
\item there is a permutation $\phi\in\Sigma_s$ such that $\delta_i=\gm_{\phi(i)}$ and $k_i=h_{\phi(i)}$,
for $i=1,\ldots,s$.
\end{enumerate}
\end{theorem}

\begin{proof}By Theorem~\ref{embedding}, it is enough to show that, for a proper characteristic $\cC$-open subgroup $K$ of $\hP^\cC$, 
a pro-$\cC$ multitwist $\tau_{\gm_1}^{h_1}\cdot\ldots\cdot\tau_{\gm_s}^{h_s}$ determines the $\Q_p$-subspace $V_{K,\sg}$ of $H_1(\ol{S}_K,\Q_p)$,
where $\sg=\{\gm_1,\ldots,\gm_s\}$. 

Let $N_{\cGG(S_K)^\cC}(G_K)$ be the normalizer of the finite group $G_K$ in the pro-$\cC$ congruence mapping class group $\cGG(S_K)$.
There is an exact sequence (cf.\ Section~2 in \cite{sym}. Note that here we are considering pure mapping class groups):
\[1\to G_K\to N_{\cGG(S_K)^\cC}(G_K)\to\cGG(S)^\cC,\]
where the image of $N_{\cGG(S_K)^\cC}(G_K)$ in $\cGG(S)^\cC$ is an open normal subgroup.

Let $U$ be an open characteristic subgroup of $N_{\cGG(S_K)^\cC}(G_K)$ such that $U\cap G_K=\{1\}$. Then, $U$
identifies with an open normal subgroup of $\cGG(S)^\cC$. Since the $U$-orbit of $\sg$ is open inside its $\cGG(S)^\cC$-orbit, the $U$-orbit of $\sg$
contains an element in the image of $C(S)$. There is then an $f\in U$ such that $f(\sg)\in C(S)$ and, for such element, there holds
$V_{K,f(\sg)}=\bar f(V_{K,\sg})$, where $\bar f$ is the image of $f$ in the symplectic group $\Sp(H_1(\ol{S}_K,\Q_p))$. Therefore, in order to prove
the above claim, it is enough to consider the case when $\sg\in C(S)\subset L(\hP^\cC)$.

There is an $r\in\N^+$ such that, for every simple closed curve $\gm$ on $S$, we have that $\tau_\gm^r\in U\cap\GG(S)$. 
From our hypothesis on the multi-index $(h_1,\ldots,h_s)$, it follows that there is an $m\in\ZZ^\cC$ such that $mh_i$ is 
a positive integer divisible by $r$, for all $i=1,\ldots,s$, and then such that $\tau_{\gm_1}^{mh_1}\cdot\ldots\cdot\tau_{\gm_s}^{mh_s}\in U\cap\GG(S)$.

Since $U\cap G_K=\{1\}$, the subgroup $U$ is contained in the centralizer $Z_{\cGG(S_K)^\cC}(G_K)$ of $G_K$ in $\cGG(S_K)$ and so
there is a natural representation $\rho_U\co U\cap\GG(S)\to Z_{\Sp(H_1(\ol{S}_K,\Q))}(G_K)$. 

In order to prove our claim and the theorem, it is then
enough to show that we can recover the subspace $V_{K,\sg}^\Q:=V_{K,\sg}\cap H_1(\ol{S}_K,\Q)$ of $H_1(\ol{S}_K,\Q)$ from the multitransvection
$\rho_U(\tau_{\gm_1}^{mh_1}\cdot\ldots\cdot\tau_{\gm_s}^{mh_s})$. The logarithm of this multitransvection is the symmetric bilinear form on 
$H_1(\ol{S}_K,\Q)$ (cf.\ Section~1 in \cite{Prym}) defined by:
\[\log(\rho_U(\tau_{\gm_1}^{mh_1}\cdot\ldots\cdot\tau_{\gm_s}^{mh_s}))=\sum_{i=1}^s\nu_i
\sum_{\td{\gm}/\gm_i}\langle[\td{\gm}],\_\rangle_{\ol{S}_K}^2\in\sym^2(H^1(\ol{S}_K,\Q)),\]
where the second sum runs over the simple closed curves $\td{\gm}$ on $S_K$ which cover $\gm_i$, we let $\nu_i:=mh_i/\deg(\td{\gm}/\gm_i)$, for 
$i=1,\ldots,s$, and we denote by $[\td{\gm}]$ a cycle in $H_1(\ol{S}_K,\Q)$ supported on $\td{\gm}$ and by $\langle\_,\_\rangle_{\ol{S}_K}$ 
the standard symplectic bilinear form on $H_1(\ol{S}_K,\Q)$. 

In Lemma~5.11 of \cite{faith}, we showed that, when $\log(\rho_U(\tau_{\gm_1}^{mh_1}\cdot\ldots\cdot\tau_{\gm_s}^{mh_s}))$ is a semidefinite 
symmetric bilinear form (e.g.\ $\nu_i\in\N^+$, for all $i=1,\ldots,s$), it is possible to recover the subspace $V_{K,\sg}^\Q$ as the \emph{core} 
(cf.\ ibidem) of this form. 

However, it is not difficult to show that this is
true also without that assumption. Indeed, for $K$ a proper characteristic subgroup of $\hP^\cC$, the set $\{\td{\gm}\}$ of simple closed curves
lying over a fixed $\gm_i$ disconnects the surface $S_K$, for every $i=1,\ldots,s$. There are then subspaces $W^+$ and $W^-$ of $H_1(\ol{S}_K,\Q)$, 
nondegenerate for the symplectic form $\langle\_,\_\rangle_{\ol{S}_K}$, with the property that 
$V_{K,\sg}^\Q=(W^+\cap V_{K,\sg}^\Q)+(W^-\cap V_{K,\sg}^\Q)$ and such that 
the restriction of  $\log(\rho_U(\tau_{\gm_1}^{mh_1}\cdot\ldots\cdot\tau_{\gm_s}^{mh_s}))$ to $W^+$ and $W^-$ is semidefinite. The same argument of 
Lemma~5.11 of \cite{faith} then applies to show that we can recover the subspaces $V_{K,\sg}\cap W^\pm$ from such restriction and hence the subspace 
$V_{K,\sg}^\Q$ as their sum.
\end{proof}

From the identity $f\cdot(\tau_{\gm_1}^{h_1}\cdot\ldots\cdot\tau_{\gm_k}^{h_k})\cdot f^{-1}=
\tau_{f(\gm_1)}^{h_1}\cdot\ldots\cdot\tau_{f(\gm_k)}^{h_k}$, it then follows at once:

\begin{corollary}\label{centralizers multitwists}
Let $\sg=\{\gm_1,\ldots,\gm_s\}$ be a pro-$\cC$ multicurve on $S$ and
$(h_1,\ldots,h_k)\in m_\sg\cdot(\Z\ssm\{0\})^k$ a multi-index, where $m_\sg\in(\ZZ^\cC)^\ast$. Then, we have
\[Z_{\cGG(S)^\cC}(\tau_{\gm_1}^{h_1}\cdot\ldots\cdot\tau_{\gm_k}^{h_k})
= N_{\cGG(S)^\cC}(\langle \tau_{\gm_1}^{h_1}\cdot\ldots\cdot\tau_{\gm_k}^{h_k}\rangle)
=N_{\cGG(S)^\cC}(\langle \tau_{\gm_1},\ldots,\tau_{\gm_k}\rangle)=\cGG(S)^\cC_\sg,\]
where $\cGG(S)^\cC_\sg$ is the stabilizer of $\sg$ described in Theorem~\ref{stab pro-curves}.
\end{corollary}

\subsection{The center of the pro-$\cC$ congruence mapping class group}\label{center}
\begin{theorem}\label{centerfree}Let $S=S_{g,n}$ be a surface of negative Euler characteristic. For every open subgroup $U$ of 
$\cGG(S)^\cC$, we then have:
\[Z_{\cGG(S)^\cC}(U)=Z(\cGG(S)^\cC)=Z(\GG(S)).\]
Thus, all these groups are trivial for $(g,n)\neq(1,1),(2,0)$. Otherwise, they are 
generated by the hyperelliptic involution. 
\end{theorem}

\begin{proof}Let us prove the identity $Z_{\cGG(S)^\cC}(U)=Z(\cGG(S)^\cC)$. Since $\cGG(S)^\cC$ is topologically generated by
a finite set of Dehn twists $\{\tau_\gm\}_{\gm\in I}$ and there is an $h_U\in\N^+$ such that $\tau_\gm^{h_U}\in U$, for all $\gm\in I$,
by Corollary~\ref{centralizers multitwists}, there is a series of identities and an inclusion:
\[Z(\cGG(S)^\cC)=Z(\langle\tau_\gm\rangle_{\gm\in I})=Z(\langle\tau_\gm^{h_U}\rangle_{\gm\in I})\supseteq Z_{\cGG(S)^\cC}(U).\]
The other inclusion $Z(\cGG(S)^\cC)\subseteq Z_{\cGG(S)^\cC}(U)$ is trivial, hence the conclusion follows.

In order to prove the identity $Z(\cGG(S)^\cC)=Z(\GG(S))$, we need to consider first the case $g=0$. We treat this case by induction on 
$n\geq 3$. The case $n=3$ is trivial because, for $n=3$, the group $\GG(S)$ is trivial. Let us then assume that the theorem holds for $3\leq n\leq k$ 
and let us prove it for $n=k+1$. Let $\gm$ be a simple closed curve on $S$ which separates two of the punctures on $S$ from the others. Then,
we have $Z(\cGG(S)^\cC)\subseteq Z_{\cGG(S)^\cC}(\tau_\gm)$. Since $Z(\GG(S))\cap\langle\tau_\gm\rangle=\{1\}$ and 
$\cGG(S)^\cC_\gm=\cGG(S)^\cC_{\vec\gm}$ (because we are considering pure mapping class groups), by Corollary~\ref{centralizers multitwists}, 
the center $Z(\cGG(S)^\cC)$ embeds in the center of $\cGG(S\ssm\gm)^\cC$. Now, the induction hypothesis implies that the latter is trivial and 
so the conclusion follows.

For $g=1,n=1$, we have to prove that $Z(\cGG(S)^\cC)=Z(\GG(S))=\langle\iota\rangle$, where $\iota$ is the elliptic involution. Let $\gm$ be a nonseparating 
closed curve on $S$. Then, by Corollary~\ref{centralizers multitwists}, $Z(\cGG(S)^\cC)\subseteq Z(\cGG^\cC_\gm)$. Given $f\in Z(\cGG(S)^\cC)$, possibly
after replacing $f$ with $f\cdot\iota$, we can assume that actually $f\in\cGG^\cC_{\vec\gm}$. But then, since 
$Z(\GG(S))\cap\langle\tau_\gm\rangle=\{1\}$, the element $f$ can be identified with an element of the center of $\cGG(S\ssm\gm)^\cC$, which, by
the genus $0$ case, is trivial. Therefore, $f$ is either trivial or equal to $\iota$, which proves the claim.

For $g=1,n\geq 2$, we have to prove that $Z(\cGG(S)^\cC)=Z(\GG(S))=\{1\}$. Let again $\gm$ be a nonseparating closed curve on $S$. Then, 
by Corollary~\ref{centralizers multitwists}, $Z(\cGG(S)^\cC)\subseteq Z(\cGG^\cC_\gm)=Z(\cGG^\cC_{\vec\gm})$. As above, by the genus $0$ case, 
this implies that the center of $\cGG(S)^\cC$ is trivial.

For $g=2$, $n=0$, we have to prove that $Z(\cGG(S)^\cC)=Z(\GG(S))=\langle\iota\rangle$, where $\iota$ is the hyperelliptic involution. Let $\gm$ be a 
nonseparating closed curve on $S$ and $f\in Z(\cGG(S)^\cC)$. As above, possibly after replacing $f$ with $f\cdot\iota$, we have that 
$f\in\cGG^\cC_{\vec\gm}$. We then proceed as in the case $g=1$, $n=1$ and show that either $f$ is trivial or $f=\iota$.
 
For $g=2$, $n\geq 1$, we have to prove that $Z(\cGG(S)^\cC)=Z(\GG(S))=\{1\}$. In this case, we let $\gm$ be a separating curve on $S$ which bounds
an unpunctured genus $1$ subsurface $S'$ of $S$. By Corollary~\ref{centralizers multitwists}, we have 
$Z(\cGG(S)^\cC)\subseteq Z(\cGG^\cC_\gm)=Z(\cGG^\cC_{\vec\gm})$. Let us show that $Z(\cGG^\cC_{\vec\gm})$ consists of the $\ZZ^\cC$-powers
of the half-twist $\tau_{1/2}$ on $S$ which restricts to the elliptic involution on $S'$ and to the identity on the complement of $S'$ in $S$. 
By the genus $1$ case, indeed, the center of $\cGG(S\ssm\gm)^\cC$ is generated by the homeomorphism $\iota$ which restricts to the elliptic involution
on $S'$ and to the identity on the complement of $S'$ in $S$. This implies the previous claim since $\tau_{1/2}^2=\tau_\gm$. The conclusion then follows
because $Z(\GG(S))\cap\langle\tau_{1/2}\rangle=\{1\}$.

For $g\geq 3$, we proceed by induction on the genus taking as basis for the induction the cases already proved. Let us then assume that the statement
of the theorem holds for all $2\leq g\leq k$ and let us prove it for $g=k+1$. Essentially the same argument of the case $g=2$, $n\geq 1$ and
the induction hypothesis yield that $Z(\cGG(S)^\cC)=Z(\GG(S))=\{1\}$.
\end{proof}

\subsection[A Birman exact sequence]{A Birman exact sequence for pro-$\cC$-congruence mapping class groups: 
the case $\check{\Pi}^\cC\equiv\hP^\cC$}\label{shortexact}
Let us now describe in a more precise way the epimorphism $\hat{\Delta}^\cC_x\co\cGG(S_x)^\cC\tura \cGG (S,x)^\cC$ introduced in 
Section~\ref{tangential} and give necessary and sufficient conditions for this to be an isomorphism. An immediate consequence of 
Corollary~\ref{centralizers multitwists} and Theorem~\ref{centerfree} is:

\begin{corollary}\label{faithful}Let $S$ be a connected surface of negative Euler characteristic, $x\in S$ and let $\check{\Pi}^\cC$ be the  
$\cC$-congruence completion of $\Pi:=\pi_1(S,x)$ defined in Section~\ref{standard}. The action by restriction of inner automorphisms of $\cGG(S_x)^\cC$ 
on its normal subgroup $\check{\Pi}^\cC$ then induces a faithful representation $\td\rho_\cC\co\cGG(S_x)^\cC\hookra\aut(\check{\Pi}^\cC)$. 
In particular, the homomorphism $\hat{\Delta}^\cC_x$ is an isomorphism, if and only if, $\check{\Pi}^\cC\equiv\hP^\cC$.
\end{corollary}

\begin{proof}Let $\gm\in\Pi\subset\check{\Pi}^\cC$ be an element whose free homotopy class contains a simple closed curve on $S$, which we also denote
by $\gm$. Then, its image in $\GG(S_x)\subset\cGG(S_x)^\cC$ is the bounding pair map $\tau_{\gm^+}\cdot\tau_{\gm^-}^{-1}$, where $\gm^+,\gm^-$ 
is a pair of simple closed curves on $S_x$ which bound a tubular neighborhood of $\gm$. Let $f\in\ker\td\rho_\cC$. By the definition of $\td\rho_\cC$,
the element $f$ centralizes the multitwist $\tau_{\gm^+}\cdot\tau_{\gm^-}^{-1}$. By Corollary~\ref{centralizers multitwists}, it then centralizes both
Dehn twists $\tau_{\gm^+},\tau_{\gm^-}^{-1}$. Since, for all $\gm\in\Pi$ as above, these elements generate (topologically) $\cGG(S_x)^\cC$,
it follows that $f\in Z(\cGG(S_x)^\cC)$. Thus, either $f=1$ or $f$ is an elliptic or hyperelliptic involution. But, in the latter case, $\td\rho_\cC(f)\neq 1$.
Therefore, $f=1$. The last claim of the corollary is obvious.
\end{proof}

\begin{remark}\label{centerfree2}Since a group acts faithfully on itself by inner automorphisms if and only if it is centerfree, 
Corollary~\ref{faithful} also implies that the profinite group $\check{\Pi}^\cC$ is centerfree.
\end{remark}

\begin{corollary}\label{comparison}Let $S$ be an oriented connected surface of finite negative Euler characteristic, $x\in S$ and
$\cC$ a class of finite groups such that $\check{\Pi}^\cC\equiv\hP^\cC$ (e.g.\ $\cC=\cF$ or $(p)$, for $p$ a prime). 
Then, $\hat{\Delta}^\cC_x\co\cGG(S_x)^\cC\tura \cGG (S,x)^\cC$ is an isomorphism
and determines a Birman exact sequence: 
\begin{equation}\label{Birman2}
1\to\hP^{\cC}\to\cGG(S_x)^\cC\to\cGG(S)^{\cC}\to 1.
\end{equation}
\end{corollary}

\begin{proof}By Corollary~\ref{faithful}, the homomorphism $\hat{\Delta}^\cC_x$ is an isomorphism. Combining this isomorphism with the 
Birman exact sequence~(\ref{Birman1}), we get the short exact sequence~(\ref{Birman2}). The claim for $\cC=\cF$ or $(p)$ is then just
(ii) of Proposition~\ref{epi}.
\end{proof}

\begin{remark}\label{Hoshi}For $\cC=\cF$ or $(p)$, Corollary~\ref{comparison} was proved by Matsumoto in the open surface case
(cf.\ Theorem~2.2 in \cite{Matsumoto}) and then extended by Hoshi and Mochizuki to the closed surface case (cf.\ Corollary~6.2 in \cite{HM} 
and Lemma~20 in \cite{Hoshi}). 
\end{remark}

\subsection{Comparison of $\cC$-congruence topologies}\label{compcong}
In Corollary~\ref{faithful}, we gave necessary and sufficient conditions under which the $\cC$-congruence topologies introduced in 
Section~\ref{standard} coincide. In this section we prove a similar result for the $\cC$-congruence topologies of Section~\ref{conftop}.
The conclusion will be that all the possible $\cC$-congruence topologies defined on the mapping class group coincide for $\cC=\cF$ or $(p)$.

With the notations of Section~\ref{conftop}, let us denote by $\pi_i\co S(n)\to S(n-1)$ the natural map which forgets the $i$-th point, for $i=1,\ldots,n$. 
The fiber of $\pi_i$ over the point $\ol{x}_i:=\{x_1,\ldots,x_{i-1},x_{i+1},\ldots,x_n\}\in S(n-1)$ is then naturally identified with $S_{\ol{x}_i}:=S\ssm\ol{x}_i$ 
and we let $\Pi_{\ol{x}_i}:=\pi_1(S_{\ol{x}_i}, x_i)$, for $i=1,\ldots,n$. There is a natural $\GG(S,\ol{x})$-equivariant embedding $\Pi_{\ol{x}_i}\hookra\Pi(n)$ 
and also $\Pi_{\ol{x}_i}\hookra\GG(S_{\ol{x}})$, for $i=1,\ldots,n$. We then denote by $\check{\Pi}_{\ol{x}_i}^\cC$ the closure of $\Pi_{\ol{x}_i}$ in the
$\cC$-congruence completion $\cGG(S_{\ol{x}})^\cC$. The following result, for $\cC=\cF$ or $(p)$, was proved by 
Mochizuki and Tamagawa (cf.\ (i) of Proposition~2.2 \cite{MT}):

\begin{proposition}\label{configuration}Let $S$ be a connected oriented surface of finite type, $S(n)$ the configuration space of $n$ distinct point on it,
let $\ol{x}=\{x_1,\ldots,x_n\}\in S(n)$ and let $\cC$ be a class of finite groups such that $\check{\Pi}_{\ol{x}_i}^\cC\equiv\hP_{\ol{x}_i}^\cC$, for all $n\geq 1$
(e.g.\ $\cC=\cF$ or $(p)$). Then, any projection morphism $\pi_i\co S(n)\to S(n-1)$ 
determines a natural short exact sequence of pro-$\cC$ fundamental groups, for $i=1,\ldots,n$:
\begin{equation}\label{conf1}
1\to\hP_{\ol{x}_i}^\cC\to\hp_1(S(n),\ol{x})^\cC\to\hp_1(S(n-1),\ol{x}_i)^\cC\to 1.
\end{equation}
\end{proposition}

\begin{proof}Let $\Pi(n):=\pi_1(S(n),\ol{x})$ and $\Pi(n-1):=\pi_1(S(n-1),\ol{x}_i)$. There is a natural embedding 
$\Pi(n)\subset\GG(S,\ol{x})\cong\GG(S_{\ol{x}})$ and we identify $\Pi(n)$ with a subgroup of $\GG(S_{\ol{x}})$. 
By induction on $n\geq 1$, we will then prove the following statement:
\begin{itemize}
\item The closure $\ol{\Pi}(n)$ of $\Pi(n)$ in the $\cC$-congruence completion $\cGG(S_{\ol{x}})^\cC$ is the pro-$\cC$ completion $\hP(n)^\cC$.
\end{itemize}

The case $n=1$ is just that $\hP(1)^\cC=\hP^\cC\equiv\ccP^\cC\subset\cGG(S_{x_1})^\cC$, which we know by hypothesis.
So, let us assume that the closure $\ol{\Pi}(n-1)$ of $\Pi(n-1)$ in $\cGG(S_{\ol{x}_i})^\cC$ is the pro-$\cC$ completion $\hP(n-1)^\cC$ and let
us prove the same statement for $\Pi(n)$.

The isomorphisms $\GG(S,\ol{x})\cong\GG(S_{\ol{x}})$ and $\GG(S,\ol{x}_i)\cong\GG(S_{\ol{x}_i})$ identify, respectively,
the group $\Pi(n)$ with a subgroup of $\GG(S_{\ol{x}})$, the group $\Pi(n-1)$ with a subgroup of $\GG(S_{\ol{x}_i})$ and the group
$\Pi_{\ol{x}_i}$ with the kernel of the natural homomorphism $\GG(S_{\ol{x}})\to\GG(S_{\ol{x}_i})$.

By hypothesis, the closure of $\Pi_{\ol{x}_i}$ in the $\cC$-congruence completion $\cGG(S_{\ol{x}_i})^\cC$ is the pro-$\cC$ completion 
$\hP_{\ol{x}_i}^\cC$. Therefore, by the inductive hypothesis and Corollary~\ref{comparison}, there is a short exact sequence:
\begin{equation}\label{conf2}
1\to\hP_{\ol{x}_i}^\cC\to\ol{\Pi}(n)\to\hP(n-1)^\cC\to 1.
\end{equation}
This immediately implies the claim that $\ol{\Pi}(n)=\hP(n)^\cC$. The statement of the proposition then follows from
the short exact sequence~(\ref{conf2}).
\end{proof}

\begin{corollary}\label{compconf}Let $\cC$ be a class of finite groups such that $\check{\Pi}_{\ol{x}_i}^\cC\equiv\hP_{\ol{x}_i}^\cC$, for all $n\geq 1$ 
(e.g.\ $\cC=\cF$ or $(p)$). Then, there are natural isomorphisms of profinite groups 
$\cGG(S,\ol{x})^\cC\cong\cGG(S_{\ol{x}_i},x_i)^\cC\cong\cGG(S_{\ol{x}})^\cC$, where $\ol{x}=\{x_1,\ldots,x_n\}\in S(n)$ and $i=1,\ldots,n$.
\end{corollary}

\begin{proof}In the proof of Proposition~\ref{configuration}, we have shown that the closure of $\Pi(n)$ in $\cGG(S_{\ol{x}})^\cC$ is
the pro-$\cC$ completion $\hP(n)^\cC$. Let us fix $i\in\{1,\ldots,n\}$. By Corollary~\ref{comparison}, there is a natural isomorphism 
$\cGG(S_{\ol{x}})^\cC\cong\cGG(S_{\ol{x}_i},x_i)^\cC$. Therefore, the closure of $\Pi(n)$ in $\cGG(S_{\ol{x}_i},x_i)^\cC$ also identifies with
the pro-$\cC$ completion $\hP(n)^\cC$. There is then a natural continuous epimorphism $\phi\co\cGG(S_{\ol{x}_i},x_i)^\cC\to\cGG(S,\ol{x})^\cC$,
obtained restricting the inner automorphisms of $\cGG(S_{\ol{x}_i},x_i)^\cC$ to its normal subgroup $\hP(n)^\cC$. On the other hand,
by Proposition~\ref{configuration}, the closure of $\Pi_{\ol{x}_i}$ in the $\cC$-congruence completion $\cGG(S,\ol{x})^\cC$ is the
pro-$\cC$ completion $\hP_{\ol{x}_i}^\cC$. Thus, by restricting the inner automorphisms of $\cGG(S,\ol{x})^\cC$ to its normal subgroup 
$\hP_{\ol{x}_i}^\cC$, we see that there is also a natural continuous epimorphism $\psi\co\cGG(S,\ol{x})^\cC\to\cGG(S_{\ol{x}_i},x_i)^\cC$.
It is clear that $\psi\circ\phi$ and $\phi\circ\psi$ are the identity when restricted to the dense subgroups $\GG(S_{\ol{x}_i},x_i)$ and $\GG(S,\ol{x})$, 
respectively. Therefore, by continuity, we have $\psi\circ\phi=\Id$ and $\phi\circ\psi=\Id$, which proves the corollary.
\end{proof}

\begin{remark}\label{obstruction}From the proof of Corollary~\ref{compconf}, it follows that if the short sequence~(\ref{conf1}) is exact, then
$\check{\Pi}_{\ol{x}_i}^\cC\equiv\hP_{\ol{x}_i}^\cC$. Since the sequence~(\ref{conf1}) is right exact for any class of finite groups $\cC$, we see that the 
obstruction to have an isomorphism $\check{\Pi}_{\ol{x}_i}^\cC\equiv\hP_{\ol{x}_i}^\cC$, for an arbitrary class $\cC$ of finite groups, is the fact that the 
sequence~(\ref{conf1}) may not be left exact. Necessary and sufficient conditions which ensure left exactness of the pro-$\cC$ completion functor are 
given by Anderson in \cite{Anderson} and are not difficult to check for $\cC=\cF$ or $(p)$ but are not necessarily satisfied by a more general class of 
finite groups (e.g.\ by $\cC=\cS$).
\end{remark}

\section[A Birman exact sequence]{A Birman exact sequence for pro-$\cC$-congruence\\ mapping class groups: the general case}\label{weakBirman}
In this Section, we will show that, for an open surface $S$, the kernel of the epimorphism $\hat{\Delta}^\cC_x\co\cGG(S_x)^\cC\tura \cGG (S,x)^\cC$
equals the kernel of the restriction of this map to $\ccP^\cC$. 

\begin{definition}\label{S-congruence2}With the notations of Section~\ref{tangential}, let $\GG(S_\circ,\dd S_\circ)$ be the mapping class group of 
self-homeomorphisms of $S_\circ$ which fix pointwise the boundary $\dd S_\circ$. The \emph{$\cC$-congruence completion 
$\cGG(S_\circ,\dd S_\circ)^\cC$ of} $\GG(S_\circ,\dd S_\circ)$ is the profinite completion induced by the $\cC$-topology on $\Pi_\circ$ 
via the faithful representation: $\rho_\circ\co\GG(S_\circ,\dd S_\circ)\hookra\aut(\Pi_\circ)$.
\end{definition}

There is a natural (in general non-split) short exact sequence:
\[1\to\tau_u^\Z\to\GG(S_\circ,\dd S_\circ)\to\GG(S_\circ)\to 1,\]
where $\tau_u$ is the Dehn twist around the boundary component $\dd S_\circ$, and a natural representation:
\[\check{\rho}^\cC\co\cGG(S_\circ,\dd S_\circ)^\cC\to\out_\cU(\hP_\circ^\cC).\]

By composing the representation $\check{\rho}^\cC$ with the isomorphism $\hat{\Kappa}_u^\cC$, defined in the proof of Lemma~\ref{tangential1},
we see that the kernel of $\check{\rho}^\cC$ is precisely the procyclic subgroup generated by the Dehn twist $\tau_u$, whose image in 
$\aut(\hP_\circ^\cC)$ is $\inn(u)$. Therefore, there is also a natural short exact sequence:
\begin{equation}\label{4.2}
1\to\tau_u^{\ZZ^\cC}\to\cGG(S_\circ,\dd S_\circ)^\cC\to\cGG(S_\circ)^\cC\to 1.
\end{equation}

Let us now assume that the Riemann surface $S_\circ$ is nonclosed and let us consider the map $p\co S_\circ\hookra S_{\circ,P}$ 
obtained filling in a puncture of $S_\circ$ with a point $P$. Let $\Pi_P:=\pi_1(S_{\circ,P},P)$. Then, there is an associated Birman short exact sequence:
\begin{equation}\label{Birman3}
1\to\Pi_P\to\GG(S_\circ,\dd S_\circ)\sr{\Psi}{\to}\GG(S_{\circ,P},\dd S_{\circ,P})\to 1.
\end{equation}
The map $p$ has a homotopy section $s\co S_{\circ,P}\hookra S_\circ$ such that $\delta:=s(\dd S_{\circ,P})$
is a separating simple closed curve on $S_\circ$ bounding, on one side, only the boundary $\dd S_\circ$ and the puncture labeled by $P$. 
Let us also fix a path from $x_\circ$ to $s(P)$, crossing only once $\delta$, so that the section $s\co S_{\circ,P}\hookra S_\circ$ induces 
a monomorphism on fundamental groups $\Pi_P\hookra\Pi_\circ$ and, by Lemma~\ref{subsurface}, also of their pro-$\cC$ completions 
$\hP_P^\cC\hookra\hP_\circ^\cC$. Let us identify the groups $\Pi_P$ and $\hP_P^\cC$ with their images in $\Pi_\circ$ and $\hP_\circ^\cC$ respectively.

The map $s$ also induces a section $s_\ast\co\GG(S_{\circ,P},\dd S_{\circ,P})\to\GG(S_\circ,\dd S_\circ)$ of the epimorphism 
$\Psi\co\GG(S_\circ,\dd S_\circ)\to\GG(S_{\circ,P},\dd S_{\circ,P})$ such that the action of the subgroup 
$s_\ast(\GG(S_{\circ,P},\dd S_{\circ,P}))$ of $\GG(S_\circ,\dd S_\circ)$ 
on $\Pi_\circ$, via the representation $\rho_\circ$, stabilizes the subgroup $\Pi_P$.
The Birman exact sequence~(\ref{Birman3}) then splits and we have a semidirect product decomposition:
\[\GG(S_\circ,\dd S_\circ)\cong\Pi_P\rtimes\GG(S_{\circ,P},\dd S_{\circ,P}).\]

\begin{theorem}\label{semidirect}Let $\ccP_P$ be the closure of the fundamental group $\Pi_P$ in the $\cC$-congruence completion 
$\cGG(S_\circ,\dd S_\circ)^\cC$. Then, the above decomposition of $\GG(S_\circ,\dd S_\circ)$ induces one:
\[\cGG(S_\circ,\dd S_\circ)^\cC\cong\ccP_P^\cC\rtimes\cGG(S_{\circ,P},\dd S_{\circ,P})^\cC.\]
In particular, there is a split Birman short exact sequence: 
\begin{equation}\label{4.3}
1\to\ccP_P^\cC\to\cGG(S_\circ,\dd S_\circ)^\cC\to\cGG(S_{\circ,P},\dd S_{\circ,P})^\cC\to 1.
\end{equation}
\end{theorem}

\begin{proof}The first remark is that, as a subgroup of $\GG(S_\circ)$, the group $s_\ast(\GG(S_{\circ,P},\dd S_{\circ,P}))$ is naturally isomorphic 
to the direct factor $\GG(s(S_{\circ,P}),\delta)$ of the stabilizer $\GG(S_\circ)_{\vec\delta}$. From Theorem~\ref{stabilizer}, it then follows that 
the closure $\ol{s_\ast(\GG(S_{\circ,P},\dd S_{\circ,P}))}$ of the subgroup $s_\ast(\GG(S_{\circ,P},\dd S_{\circ,P})$ in $\cGG(S_\circ,\dd S_\circ)^\cC$ 
is naturally isomorphic to the $\cC$-congruence completion $\cGG(S_{\circ,P},\dd S_{\circ,P})^\cC$ of 
the mapping class group $\GG(S_{\circ,P},\dd S_{\circ,P})$ associated to the natural representation in $\aut(\Pi_P)$ (cf.\ Definition~\ref{S-congruence2}).

Let us denote by $\tGG(S_{\circ,P},\dd S_{\circ,P})^\cC$ the profinite completion of $\GG(S_{\circ,P},\dd S_{\circ,P})$ induced by the
$\cC$-congruence completion $\cGG(S_\circ,\dd S_\circ)^\cC$ via the epimorphism $\Psi$ in the Birman exact sequence~(\ref{Birman3})
and by $\check{\Psi}^\cC\co\cGG(S_\circ,\dd S_\circ)^\cC\to\tGG(S_{\circ,P},\dd S_{\circ,P})^\cC$
the induced epimorphism. In order to prove the theorem, it is then enough to show that the restriction of $\check{\Psi}^\cC$ to
$\ol{s_\ast(\GG(S_{\circ,P},\dd S_{\circ,P}))}$ is an isomorphism.

Let $\tGG(S_{\circ,P})^\cC$ be the profinite completion of $\GG(S_{\circ,P})$ induced by the $\cC$-congruence completion $\cGG(S_\circ)^\cC$ 
via the natural epimorphism $\GG(S_\circ)\to\GG(S_{\circ,P})$ (with kernel $\Pi_P$). By the various definitions involved, there is then a short exact sequence:
\[1\to\tau_u^{\ZZ^\cC}\to\tGG(S_{\circ,P},\dd S_{\circ,P})^\cC\to\tGG(S_{\circ,P})^\cC\to 1.\]

After moding out the closure of $\Pi_P$ in both domain and codomain of the epimorphism 
$\hat{\Delta}^\cC_P\co\cGG(S_\circ)^\cC\tura \cGG (S_{\circ,P},P)^\cC$, we get an epimorphism 
$\bar{\Delta}^\cC_P\co\tGG(S_{\circ,P})^\cC\tura\cGG(S_{\circ,P})^\cC$ and, after moding out the normal subgroup $\tau_u^{\ZZ^\cC}$ in
both domain and codomain of the epimorphism 
$\check{\Psi}^\cC\co\ol{s_\ast(\GG(S_{\circ,P},\dd S_{\circ,P}))}\tura\tGG(S_{\circ,P},\dd S_{\circ,P})^\cC$, we get an epimorphism
$\cGG(S_{\circ,P})^\cC\tura\tGG(S_{\circ,P})^\cC$ which, composed with $\bar{\Delta}^\cC_P$, is an isomorphism and then is itself an isomorphism. 
Therefore, the restriction of $\check{\Psi}^\cC$ to $\ol{s_\ast(\GG(S_{\circ,P},\dd S_{\circ,P}))}$ is an isomorphism as well.
\end{proof}

By moding out the subgroup $\tau_u^{\ZZ^\cC}$ from the short exact sequence~(\ref{4.3}), we get:

\begin{corollary}\label{comparison2}Let $\cC$ be a class of finite groups and $S$ a nonclosed connected surface of negative Euler characteristic.
Then, the natural epimorphism $\hat{\Delta}^\cC_x\co\cGG(S_x)^\cC\tura \cGG (S,x)^\cC$ has the same kernel of the restriction 
$\hat{\Delta}^\cC_x|_{\ccP^\cC}$ and there is a short exact sequence:
\begin{equation}\label{Birman4}
1\to\ccP^\cC\to\cGG(S_x)^\cC\to\cGG(S)^\cC\to 1.
\end{equation}
\end{corollary}

\subsection{A $\cC$-congruence subgroup property for genus~$0$ mapping class groups}\label{genus 0}
In the genus $0$ case, the results of the previous section imply the weak version of the subgroup $\cC$-congruence property, for $\cC$ 
an arbitrary class of finite groups, and the standard version for $\cC=\cF$ or $(p)$ (cf.\ Definition~\ref{weak_strong}):

\begin{corollary}\label{congruence_0}Let $\cC$ be a class of finite groups, $S$ a connected surface of genus $0$ and negative Euler characteristic 
and $x$ a point on $S$.
\begin{enumerate}
\item There are natural epimorphism $\cGG(S_x)^\cC\tura\hGG(S_x)^\cC$ and $\cGG(S,x)^\cC\tura\hGG(S,x)^\cC$.
\item Let $\cC$ be such that $\check{\Pi}^\cC\equiv\hP^\cC$ for all surfaces $S$ of genus $0$ (e.g.\ $\cC=\cF$ or $(p)$). 
Then, we have $\cGG(S_x)^\cC\equiv\hGG(S_x)^\cC\equiv\cGG(S,x)^\cC\equiv\hGG(S,x)^\cC$.
\end{enumerate}
\end{corollary}

\begin{proof}By the Birman short exact sequences~(\ref{Birman1}) and (\ref{Birman4}), 
there is a commutative diagram with exact rows: 
\[\begin{array}{ccccccccl}
1&\to&\ccP^\cC&\to&\cGG(S_x)^\cC&\to&\cGG(S)^\cC&\to&1\\
&&\da\hspace{0.3cm}&&\da\sst{\hat{\Delta}^\cC_x} &&\parallel\,\,&&\\
1&\to&\hP^{\cC}&\to&\cGG(S,x)^\cC&\to&\cGG(S)^{\cC}&\to&1,
\end{array}\]
where all vertical maps are surjective. Induction on $-\chi(S)$, starting with $S$ a $3$-punctured sphere ($-\chi(S)=1$), 
for which $\GG(S)=\{1\}$, then yields both items of the corollary.
\end{proof}

\begin{remark}\label{mamoru}For $\cC=\cF$ and $\cC=(p)$, item (ii) of Corollary~\ref{congruence_0} was already proved by Asada in \cite{Asada}.
\end{remark}

\section[$\cC$-congruence subgroup properties]{$\cC$-congruence subgroup properties for hyperelliptic mapping class groups}\label{congruence}
The proof of Theorem~\ref{main} is based on the interpretation of
hyperelliptic mapping class groups as fundamental groups of moduli stacks of complex hyperelliptic curves. 
Let us recall a few facts from Section~3 of \cite{hyp}. The moduli stack of $n$-pointed, genus $g$ smooth hyperelliptic curves $\cH_{g,n}$
can be described in terms of moduli of pointed genus $0$ curves. Indeed,
there is a natural $\Z/2$-gerbe $\cH_{g}\to\cM_{0,[2g+2]}$, for $g\geq 2$, defined assigning, to a 
genus $g$ hyperelliptic curve $C$, the genus zero quotient curve $C/\iota$ marked by the branch locus of the finite morphism 
$C\to C/\iota$, where $\iota$, as usual, denotes the hyperelliptic involution. Similarly,
in the genus $1$ case, there is a $\Z/2$-gerbe $\cM_{1,1}\to\cM_{0,1[3]}$, where, by the notation "$1[3]$", we mean that 
one label is distinguished from the others (which instead are unordered). 

For $2g-2+n>0$, there is also a natural representable morphism $\cH_{g,n+1}\to\cH_{g,n}$, forgetting the 
$(n+1)$-th marked point, which is isomorphic to the universal $n$-punctured curve over $\cH_{g,n}$.
Thus, the fibre above an arbitrary closed point $x\in\cH_{g,n}$ is diffeomorphic to $S_{g,n}$ and its
fundamental group is isomorphic to $\Pi_{g,n}$. These morphisms induce, on 
topological fundamental groups, short exact sequences, for $g\geq 2$:
\begin{equation}\label{hyp1}
1\to\Z/2\to \UU_g\to\GG_{0,[2g+2]}\to 1\,\,\,\,
\mbox{ and }\,\,\,\,
1\to\Pi_{g,n}\to \UU_{g,n+1}\to \UU_{g,n}\to 1.
\end{equation}
For the algebraic fundamental groups, there are then short exact sequences (here and in the sequel, in order to simplify the notations, 
we omit to mention base points):
\begin{equation}\label{hyp2}
1\!\to\!\Z/2\!\to\!\pi_1^\mathrm{alg}(\cH_g)\!\to\!\pi_1^\mathrm{alg}(\cM_{0,[2g+2]})\!\to\! 1\,\,
\mbox{ and }\,\,
1\!\to\!\hP_{g,n}\!\to\! \pi_1^\mathrm{alg}(\cH_{g,n+1})\!\to\!\pi_1^\mathrm{alg}(\cH_{g,n})\!\to\! 1,
\end{equation}
where, for all $n\geq 0$, the groups $\pi_1^\mathrm{alg}(\cH_{g,n})$ and $\pi_1^\mathrm{alg}(\cM_{0,[2g+2]})$ identify, respectively, with 
the profinite completions $\hUU_{g,n}$ of $\UU_{g,n}$ and $\hGG_{0,[2g+2]}$ of $\GG_{0,[2g+2]}$. For $g=1$, we just need to replace the short exact
sequences which appear on the left in (\ref{hyp1}) and (\ref{hyp2}) above with $1\to\Z/2\to \GG_{1,1}\to\GG_{0,1[3]}\to 1$ and 
$1\to\Z/2\to\pi_1^\mathrm{alg}(\cM_{1,1})\to\pi_1^\mathrm{alg}(\cM_{0,1[3]})\to 1$, respectively.

The outer representation $\hat{\varrho}_{g,n}\co\pi_1^\mathrm{alg}(\cH_{g,n})\to\out(\hP_{g,n})$, associated to the left
short exact sequence in (\ref{hyp2}), is the algebraic monodromy representation of the punctured 
universal curve over $\cH_{g,n}$. The $\cF$-congruence subgroup property for the hyperelliptic mapping class group
$\UU_{g,n}$ is then equivalent to the faithfulness of $\hat{\varrho}_{g,n}$ (cf.\ \cite{hyp}).

For $\cC$ a class of finite groups, the kernel of the natural epimorphism $\hP_{g,n}\tura\hP_{g,n}^\cC$ is characteristic. 
Therefore, the representation $\hat{\varrho}_{g,n}$ induces a representation:
\[\hat{\varrho}_{g,n}^\cC\co\pi_1^\mathrm{alg}(\cH_{g,n})\to\out(\hP_{g,n}^\cC),\]
which we call the \emph{universal hyperelliptic pro-$\cC$ monodromy representation}. The image of this representation is isomorphic to
the $\cC$-congruence completion $\cUU_{g,n}^\cC$. Its kernel then identifies with the intersection $\cap_{\ld\in\Ld}U_\ld$ of all $\cC$-congruence 
open subgroups of $\hUU_{g,n}\cong\pi_1^\mathrm{alg}(\cH_{g,n})$. In particular, the kernel of $\hat{\varrho}_{g,n}^\cC$ does not change if we restrict 
this representation to a $\cC$-congruence open subgroup $U_\ld$ of $\hUU_{g,n}$.

For $\Z/2\in\cC$, the profinite completion $\hUU(2)$ of the abelian level $\UU(2)$ of $\UU_{g,n}$ is a $\cC$-congruence open subgroup of $\hUU_{g,n}$. 
Let then $\cK^\cC_{g,n}$ be the (characteristic) kernel of the natural epimorphism from the profinite completion to the pro-$\cC$ completion
$\hUU(2)\to\hUU(2)^\cC$. 

Item (i) of Theorem~\ref{main} is equivalent to the statement that, if $\Z/2\in\cC$, the epimorphism $\hUU(2)\to\hUU(2)^\cC$ factors through the 
representation $\hat{\varrho}_{g,n}^\cC$ and an epimorphism:
\[\cUU(2)^\cC\tura\hUU(2)^\cC\hspace{1cm}\mbox{ or equivalently }\hspace{1cm}\ker\hat{\varrho}_{g,n}^\cC\subseteq\cK^\cC_{g,n}.\]

Item (ii) of Theorem~\ref{main} is equivalent to the statement that, under the given hypothesis, the epimorphism $\hUU(2)\to\hUU(2)^\cC$ 
factors through the representation $\hat{\varrho}_{g,n}^\cC$ and an isomorphism:
\[\cUU(2)^\cC\sr{\sim}{\to}\hUU(2)^\cC\hspace{1cm}\mbox{ or equivalently }\hspace{1cm}\ker\hat{\varrho}_{g,n}^\cC=\cK^\cC_{g,n}.\]

The \'etale covering $\cH^\ld\to\cH_{g,n}$ associated to a (geometric) level $\UU^\ld$ of $\UU_{g,n}$ is called
\emph{a (geometric) level structure} over $\cH_{g,n}$. The abelian level structure $\cH^{(m)}$ of order $m$, for $m\geq 2$,
is the covering associated to the abelian level $\UU(m)$ of $\UU_{g,n}$.
There is a standard procedure to simplify the structure of an algebraic stack $X$ by 
erasing a generic group of automorphisms $G$. The algebraic stack thus obtained is usually denoted by $X\!\!\fatslash G$. 
The natural map $\cH_g\to\cM_{0,[2g+2]}$ yields an isomorphism 
$\cH_g\!\!\fatslash\langle\iota\rangle\cong\cM_{0,[2g+2]}$. In Proposition~3.3 in \cite{hyp}, it was observed that the pull-back over $\cH_g$ 
of the Galois \'etale covering $\cM_{0,2g+2}\to\cM_{0,[2g+2]}$ is the abelian level structure $\cH_g^{(2)}$. More precisely:

\begin{proposition}\label{ordered}For $g\geq 2$, there is a natural isomorphism 
$\cH^{(2)}\!\!\!\fatslash\langle\iota\rangle\cong\cM_{0,2g+2}$. In particular, the abelian level of order $2$ of $\UU_g$ is described 
by the short exact sequence:
\begin{equation}\label{abelian_2}
1\to\langle\iota\rangle\to\UU(2)\to\GG_{0,2g+2}\to 1.
\end{equation}
\end{proposition}

\begin{remark}Similarly, for the abelian level structure $\cM^{(2)}$ over $\cM_{1,1}$, there is a natural isomorphism 
$\cM^{(2)}\!\!\!\fatslash\langle\iota\rangle\cong\cM_{0,4}$, where $\iota$ here denotes the generic elliptic involution
and a short exact sequence $1\to\langle\iota\rangle\to\GG(2)\to\GG_{0,4}\to 1$. From now on, for simplicity of exposition, 
we will focus on the moduli spaces of hyperelliptic curves $\cH_{g,n}$, for $g\geq 2$, 
and leave to the reader the obvious reformulation and proof of the similar statements for the case $g=1$ and $n\geq 1$.
\end{remark}

Let $\cC_g\to\cH_g$, for $g\geq 2$, be the universal curve. Removing Weierstrass points from its fibres,
we obtain a $(2g+2)$-punctured, genus $g$ curve $\cC_0\to\cH_{g}$. Let $C_0$ be the fiber of $\cC_0\to\cH_{g}$ over the base point of 
the algebraic fundamental group $\hUU_g$ of $\cH_g$ and let us denote by 
\begin{equation}\label{rho_0}
\hat{\varrho}_0^\cC\co\hUU_g\hookra\out(\hp_1(C_0)^\cC)
\end{equation}
the associated algebraic monodromy representation.

By Proposition~\ref{ordered}, since $\iota\notin\UU(4)$, there is a natural representable \'etale Galois morphism 
$\cH^{(4)}\to\cM_{0,2g+2}$. Let $\cC_0^{(4)}\to\cH^{(4)}$ and ${\cal R}^{(4)}\to\cH^{(4)}$ be, respectively, the pull-backs 
of the curve $\cC_0\to\cH_g$ and of the universal $(2g+2)$-punctured, genus $0$ curve ${\cal R}\to\cM_{2g+2}$ over $\cH^{(4)}$.
There is then a commutative diagram:
\[\begin{array}{ccc}
\cC_0^{(4)}&\sr{\psi}{\to}&{\cal R}^{(4)}\\
&\searrow&\da\hspace{0.4cm}\\
&& \cH^{(4)},
\end{array}
\]
where $\psi$ is the \'etale, degree $2$ map which, fibrewise, is the natural map to the quotient by the hyperelliptic involution. 
Pick a base point of $ \cH^{(4)}$ so that the fiber over it of the curve $\cC_0^{(4)}\to\cH^{(4)}$ identifies with $C_0$ and
let $R$ be the fiber over the chosen base point of the curve ${\cal R}^{(4)}\to\cH^{(4)}$.
From the above diagram, we then get the algebraic monodromy representations: $\hUU(4)\to\out(\hp_1(C_0)^\cC)$, which is just the 
restriction of $\hat{\varrho}_0^\cC$ to $\hUU(4)$, has its same kernel and we denote in the same way, and 
$\hat{\varrho}_{\cal R}^\cC\co\hUU(4)\to\out(\hp_1(R)^\cC)$. We have:

\begin{lemma}\label{weak}If $\Z/2\in\cC$, the representation $\hat{\varrho}_0^\cC$~(\ref{rho_0}) has the same kernel of $\hat{\varrho}_{\cal R}^\cC$ 
so that the short exact sequence~(\ref{abelian_2}) induces one
\begin{equation}\label{abelian_2_pro}
1\to\langle\iota\rangle\to\hat{\varrho}_0^\cC(\hUU(2))\to\cGG_{0,2g+2}^\cC\to 1.
\end{equation}
Moreover, we have:
\begin{enumerate}
\item $\ker\hat{\varrho}_0^\cC\subseteq\cK^\cC_g$.
\item If the hypothesis of (ii) of Theorem~\ref{main} are satisfied, then $\ker\hat{\varrho}_0^\cC=\cK^\cC_g$.
\end{enumerate}
\end{lemma}

\begin{proof}Before we start the proof of the lemma, let us remark that the subgroup $\cK^\cC_g$ of $\hUU(2)$  is contained in any
$\cC$-open subgroup of this group and thus, in particular, is contained in $\hUU(4)$.
In fact, the level $\UU(4)$ has index a power of $2$ in $\UU(2)$ and hence its profinite completion $\hUU(4)$
is a $\cC$-open normal subgroup of $\hUU(2)$.

Since $\Z/2\in\cC$, the group $\hp_1(C_0)^\cC$ identifies with an open normal subgroup of $\hp_1(R)^\cC$
of index $2$. The inclusion $\ker\hat{\varrho}_0^\cC\subseteq\ker\hat{\varrho}_{\cal R}^\cC$ then follows from:

\begin{lemma}\label{asada}Let $\wh{F}^\cC$ be a free pro-$\cC$ group of finite rank and $N$ an open normal subgroup of $\wh{F}^\cC$.
If $\alpha$ is an automorphism of $\wh{F}^\cC$ which restricts to the identity on $N$, then $\alpha$ is the identity automorphism.
\end{lemma}

\begin{proof}The proof is essentially the same as the proof of Lemma~8 in \cite{Asada}. The only difference is that, where Asada appeals to 
his identity $(1.3.1)$, we need to use Lemma~\ref{central}.
\end{proof}

To see the other inclusion $\ker\hat{\varrho}_{\cal R}^\cC\subseteq\ker\hat{\varrho}_0^\cC$, let us observe that, if $f\in\ker\hat{\varrho}_{\cal R}^\cC$,
the element $\hat{\varrho}_0^\cC(f)$ acts on $\hp_1(C_0)^\cC$ through a power of the hyperelliptic involution. But then 
$\hat{\varrho}_0^\cC(f)=1$, since $\hat{\varrho}_0^\cC(f)$ acts trivially on $H_1(C_0,\Z/4)$ while the hyperelliptic involution acts
on $H_1(C_0,\Z/4)$ by multiplication by $-1$.

Now, since $\iota\notin\hUU(4)$, the natural epimorphism $\hUU(2)\tura\hGG_{2g+2,0}$ identifies $\hUU(4)$ with a $\cC$-open subgroup of 
$\hGG_{2g+2,0}$ and then the subgroup $\cK^\cC_g$ of $\hUU(4)$ is identified with the kernel of the natural epimorphism 
$\hGG_{2g+2,0}\tura\hGG_{2g+2,0}^\cC$. The short exact sequence~(\ref{abelian_2_pro}) and items (i), (ii) of the lemma now follow 
from Corollary~\ref{congruence_0}, Proposition~\ref{ordered} and the part of the lemma which we already proved.
\end{proof}

We can now prove a special case of Theorem~\ref{main}:

\begin{lemma}\label{final step}For $\cC$ a class of finite groups such that $\Z/2\in\cC$, we have:
\begin{enumerate}
\item $\ker\hat{\varrho}_{g,2g+2}^\cC\subseteq\cK^\cC_{g,2g+2}$.
\item If the hypothesis of (ii) of Theorem~\ref{main} are satisfied, then $\ker\hat{\varrho}_{g,2g+2}^\cC=\cK^\cC_{g,2g+2}$.
\end{enumerate}
\end{lemma}

\begin{proof}The universal curve $\cC^{(2)}\to\cH_g^{(2)}$ over the abelian level structure of order $2$ is endowed with $2g+2$ sections 
corresponding to the Weierstrass points on the fibres. Let us fix an order on this set of sections. 
By the universal property of the level structure $\cH_{g,n}^{(2)}$, there is then a morphism 
$s\co\cH_g^{(2)}\to\cH_{g,2g+2}^{(2)}$, which is a section of the natural projection 
$p\co\cH_{g,2g+2}^{(2)}\to\cH_g^{(2)}$ (forgetting the labels). The morphism $p$ is 
smooth and its fibre above a closed point $[C]\in\cH_g$ is the configuration space 
of $2g+2$ points on the curve $C$. 

Let us denote by $S_g(n)$ the configuration space of
$n$ points on the closed Riemann surface $S_g$ and by $\Pi_g(n)$ its fundamental group.
All fibres of $p$ above closed points of $\cH_g$ are diffeomorphic to $S_g(2g+2)$. Therefore,
we have a semidirect product decomposition of the fundamental group $\UU_{g,2g+2}(2)$ of $\cH_{g,2g+2}^{(2)}$:
\[\UU_{g,2g+2}(2)=\Pi_g(2g+2)\rtimes\UU_g(2).\]
In Section~3 of \cite{hyp}, we saw that this decomposition induces one:
\[\hUU_{g,2g+2}(2)\cong\hP_g(2g+2)\rtimes\hUU_g(2).\]

For $n\geq 1$, let us denote by $\ccP_g(2g+2)^\cC$ the closure of $\Pi_g(2g+2)$ in the $\cC$-congruence completion $\cUU_{g,2g+2}^\cC$.
By Corollary~\ref{comparison2}, for all $n\geq 1$, there is then a short exact sequence:
\begin{equation}\label{hypconf}
1\to\ccP_{g,n}^\cC\to\ccP_g(n+1)^\cC\to\ccP_g(n)^\cC\to 1.
\end{equation}
Hence, by induction on $n$ and (i) of Proposition~\ref{epi}, we see that, for all $n\geq 1$, there is a natural epimorphism $\ccP_g(n)^\cC\to\hP_g(n)^\cC$.

By Lemma~\ref{weak}, in order to prove Lemma~\ref{final step}, it is then enough to show that the $\cC$-congruence completion
of $\UU_{g,2g+2}(2)$ admits a semidirect product decomposition:
\begin{equation}\label{hypdec}
\cUU_{g,2g+2}(2)^\cC\cong\ccP_g(2g+2)^\cC\rtimes\hat{\varrho}_0^\cC(\hUU_g(2)).
\end{equation}
To prove the existence of the decomposition~(\ref{hypdec}), we have to show that $\ccP_g(2g+2)^\cC$ has trivial intersection with 
$\hat{\varrho}_0^\cC(\hUU_g(2))$ inside the pro-$\cC$ congruence hyperelliptic mapping class group $\cUU_{g,2g+2}(2)^\cC$. 

As in the proof of Lemma~3.7 in \cite{hyp}, the key observation is that the group $\hat{\varrho}_0^\cC(\hUU_g(2))$ centralizes
the hyperelliptic involution $\hat{\varrho}_0^\cC(\iota)$, which we simply denote by $\iota$ (as well all the images of this element in the groups
$\cUU_{g,n}(2)^\cC$, for $0\leq n<2g+2$). To conclude, we then have to show that instead the centralizer of $\iota$ in $\ccP_g(2g+2)^\cC$ is trivial.
By the short exact sequence~(\ref{hypconf}) and induction on $n$, this last claim is implied by the lemma:

\begin{lemma}\label{orbifold}For all $0\leq n\leq 2g+1$ (where, for $g=1$, we let $n\geq 1$), the subgroup $\langle\iota\rangle$ is self-normalizing in 
$\ccP_{g,n}^\cC\cdot\langle\iota\rangle$.
\end{lemma}

\begin{proof}The subgroup $\Pi_{g,n}\cdot\langle\iota\rangle$ of the hyperelliptic mapping class group $\UU_{g,n+1}$ identifies with the 
fundamental group of the orbifold quotient by the hyperelliptic involution of the punctured hyperelliptic Riemann surface obtained from a closed
hyperelliptic Riemann surface $(S_g,\iota)$ of genus $g$ removing $n$ of its $2g+2$ Weierstrass points.

The first remark is then that we can reduce to the case $n=0$. In fact, $\Pi_{g,n}$ identifies with a subgroup of the fundamental group $\pi_1(S,x)$ of 
the closed hyperelliptic Riemann surface $S$ of genus $g+n$ obtained from $S_{g,n}$ capping each of the $n$ punctures with a surface of genus $1$, 
where we let $x\in S_{g,n}\subset S$ be one of the Weierstrass points of $S_g$ which has not been removed.
An immediate consequence of Theorem~\ref{stabilizer} is then that the closure of 
$\Pi_{g,n}$ in the $\cC$-congruence completion $\cGG(S_x)^\cC$ is $\ccP_{g,n}^\cC$, where we identify $\pi_1(S,x)$ with the kernel of the 
natural epimorphism $\GG(S_x)\to\GG(S)$ and $\Pi_{g,n}$ with a subgroup of $\pi_1(S,x)$ as explained above.
For $n=0$, Lemma~\ref{orbifold} and then Lemma~\ref{final step} follow from the lemma:

\begin{lemma}\label{chevalley}Let $[S/\iota]$ be the orbifold quotient of a hyperelliptic closed Riemann surface $S$ by the hyperelliptic involution $\iota$ 
and $\pi\co S_N\to[S/\iota]$ the unramified covering associated to a finite index normal proper subgroup $N$ of $\Pi:=\pi_1(S,x)$ preserved by $\iota$, 
so that the covering transformation group of $\pi$ is naturally isomorphic to $G:=G_N\rtimes\langle\iota\rangle$, where $G_N:=\Pi/N$. 
Then, $S_N$ is not hyperelliptic and $\iota$ is self-centralizing in $G$.
\end{lemma}

\begin{proof}We will show that the lemma is a consequence of the representation theory of the $G$-module $H^1(S_N,\C)$. By restriction, this is also
a $G_N$-module and its $G$-module and $G_N$-module structures are related in the following way (cf.\ Proposition~5.1 in \cite{FH}). Let $V\neq\{0\}$ 
be an irreducible $G$-submodule of $H^1(S_N,\C)$, there are then two possibilities: 
\begin{enumerate}
\item $V$ is also irreducible as $G_N$-module and $\iota$ acts on it either trivially or by multiplication by $-1$;
\item $V$ is the sum of two irreducible and not isomorphic $G_N$-submodules $W$ and $\iota(W)$. 
\end{enumerate}
Moreover, every irreducible $G$-submodule of $H^1(S_N,\C)$ is of one of the two above types. Let us show that the first possibility (i) occurs only 
when $V$ is a trivial $G_N$-submodule of $H^1(S_N,\C)$ on which $\iota$ acts by multiplication by $-1$. 

Let us consider first the case when $\iota$ acts trivially on $V$ and then when it acts by multiplication by $-1$. 
Let $Q_1,\ldots,Q_{2g+2}$ be the branch points of the covering $\pi\co S_N\to S/G$ and $G_{P_i}$ the corresponding stabilizer of a ramification point 
$P_i$ lying over $Q_i$, for $i=1,\ldots, 2g+2$, where $g\geq 2$ is the genus of $S$. These are cyclic groups of order $2$ generated by some conjugate 
of $\iota$. If $V$ is a nontrivial $G$-module, its multiplicity $\mu$ in $H^1(S_N,\C)$ is given by the formula (see, for instance, Corollary to Theorem~2 
and formula (5) in Section~1 of \cite{Kani}, but the formula can also be proven with elementary methods):
\[\mu=2\dim V(g(S_N/N)-1)+\sum_{i=1}^{2g+2}(\dim V-\dim V^{G_{P_i}}),\]
where $g(S_N/N)$ is the genus of the quotient $S_N/N$. If the $G$-module $V$ is trivial, its multiplicity, of course, is simply $\mu=2g(S_N/N)$.
Therefore, in case $\iota$ acts trivially on $V$, we have that $\dim V-\dim V^{\iota}=0$. By the above formula, since $g(S_N/N)=0$,
the multiplicity of the irreducible $G$-module $V$ in $H^1(S_N,\C)$ is then $-2\dim V$, an absurd. 

Let us then consider the case when $\iota$ acts by multiplication by $-1$ on $V$.
We have that $\dim V-\dim V^{\iota}=\dim V$, which, by the same formula, implies
that the irreducible $G$-module $V$ has multiplicity $2g\cdot\dim V$ in $H^1(S_N,\C)$. However, by our assumptions, the group $G_N$ operates 
freely on $S_N$. Hence, the same formula, applied to the covering $p_N\co S_N\to S$, implies that, as a $G_N$-module, $H^1(S_N,\C)$ consists of 
$2g-2$ copies of the regular representation plus two copies of the trivial $G_N$-module. Therefore, $\iota$ acts by multiplication by $-1$ 
exactly when $V$ is the trivial $G_N$-module, i.e.\ $V\subset H^1(S_N/G_N,\C)$. This also shows that $S_N$ is not hyperelliptic and, in particular,
$H^1(S_N,\C)^\iota\neq\{0\}$.

Let us now suppose that $\alpha\neq 1$ is an element of $G_N$ which commutes with $\iota$ and let $C_\alpha$ the cyclic subgroup of $G$ generated
by $\alpha$. Then $H^1(S_N,\C)^\iota$ is a $C_\alpha$-module and so it contains an irreducible $C_\alpha$-submodule $W\neq\{0\}$.
Since $C_\alpha\subseteq G_N$, there is an irreducible $G_N$-submodule $V$ of $H^1(S_N,\C)$ such that $W\subseteq V$. Now,
$\iota(V)\cap V\supseteq W$. Therefore, from the above discussion, it follows that $\iota(V)=V$ and $\iota$ acts on it by multiplication by $-1$, 
which cannot be because $V\supseteq W$. Thus, we conclude that $Z_G(\iota)=\langle\iota\rangle$.
\end{proof}
\end{proof}
\end{proof}

\begin{proof}[Proof of Theorem~\ref{main}](i): By the Birman exact sequence~(\ref{Birman4}), for $n\geq 1$, there is a commutative diagram with exact rows
\[\begin{array}{ccccccc}
1\to&\hP_{g,n}&\to&\hUU_{g,n+1}(2)&\to&\hUU_{g,n}(2)&\to 1\,\\
&\da\,\,\,\,\,&&\da{\st \hat{\varrho}_{g,n+1}^\cC}&&\,\,\da{\st \hat{\varrho}_{g,n}^\cC}&\\
1\to&\ccP_{g,n}^\cC&\to&\cUU_{g,n+1}(2)^\cC&\to&\cUU_{g,n}(2)^\cC&\to 1,
\end{array}\]
where the epimorphism $\hP_{g,n}\tura\hP_{g,n}^\cC$ factors through the epimorphism $\ccP_{g,n}^\cC\tura\hP_{g,n}^\cC$
(cf.\ (i) of Proposition~\ref{epi}). Proceeding by ascending and then descending induction on $n$, we see that the natural epimorphism 
$\hUU_{g,n}(2)\tura\hUU_{g,n}(2)^\cC$ factors through $\hat{\varrho}_{g,n}^\cC$ and an epimorphism $\cUU_{g,n}(2)^\cC\tura\hUU_{g,n}(2)^\cC$ 
both for $n\geq 2g+2$ and for $1\leq n\leq 2g+2$.
\smallskip

\noindent
(ii): The hypothesis $\ccP_{g,n}^\cC\equiv\hP_{g,n}^\cC$ for all $n\geq 0$ and (ii) of Lemma~\ref{final step}
immediately imply that $\cUU_{g,n}(2)^\cC$ is a pro-$\cC$ group for all $n\geq 0$. Therefore,
by the Birman exact sequence~(\ref{Birman2}), there is a commutative diagram with exact rows:
\[\begin{array}{ccccccc}
1\to&\hP_{g,n}^\cC&\to&\hUU_{g,n+1}(2)^\cC&\to&\hUU_{g,n}(2)^\cC&\to 1\,\\
&\|\,\,\,\,\,&&\da\,\,&&\,\,\da\,\,&\\
1\to&\hP_{g,n}^\cC&\to&\cUU_{g,n+1}(2)^\cC&\to&\cUU_{g,n}(2)^\cC&\to 1.
\end{array}\]
Proceeding by induction, we see that the natural epimorphism $\hUU_{g,n}(2)^\cC\tura\cUU_{g,n}(2)^\cC$ is an isomorphism for
both $n\geq 2g+2$ and $0\leq n\leq 2g+2$.
\end{proof}

\section{Orbits of simple elements in a pro-$\cC$ surface group}\label{orbits}
In this section, we will prove a geometric version of a conjecture by Gelander and Lubotzky. Let us recall that an element of a free group is \emph{primitive},
if it belongs to some minimal generating set. For $F$ a free group of finite rank, the conjecture states that the intersection with $F$ of the 
$\aut(\hat{F})$-orbit of a primitive element of $F$ (where we identify $F$ with a subgroup of its profinite completion $\hat{F}$) consists of primitive elements. 
Otherwise stated, the set of primitive elements of $F$ is closed in the profinite topology.
This conjecture has been proved by Puder and Parzanchevski (cf.\ Corollary~8.1 in \cite{PP}). 

Shortly after the publication of \cite{scc}, Pavel Zalesskii and I realized that (ii) of Theorem~4.3 and Proposition~4.1 of that paper imply a 
strong geometric version of the above conjecture for closed surface groups. For a surface group $\Pi=\pi_1(S,x)$ (where the surface $S$ is 
of finite type but not necessarily closed), the notion which roughly corresponds to that of primitive element is that of 
\emph{simple element}. An element of $\Pi$ is \emph{simple} if its free homotopy class contains a simple closed curve. 

An immediate consequence of the aforementioned results is then that, for any class of finite groups $\cC$, the set of simple elements of a 
closed surface group $\Pi$ is closed in the pro-$\cC$ topology on $\Pi$. In this section, this result will be extended to any connected surface $S$ of negative
Euler characteristic. Moreover, we will show that, for $\cC=\cF$ or $(p)$, the orbits of simple elements by $\cC$-congruence subgroups of $\GG(S,x)$ 
are also closed for the pro-$\cC$ topology on $\Pi$. 

\subsection{The pro-$\cC$ Reidemeister pairing}
In order to proceed, we need first to extend the results of Section~4 in \cite{BZ} to an arbitrary surface group. Let, as usual, $S$ be a surface of finite 
negative Euler characteristic and let $\Pi$ be its fundamental group. It is then clear how to define a pro-$\cC$ Reidemeister pairing:
\[{\mathfrak R}_\cC\co\hP^\cC\times\hP^\cC\to\Q_p[[\hP^\cC]],\]
which is continuous and invariant under the natural action by conjugation of the group $\hP^\cC$ on the domain.
We just need to follow step by step the definitions given in Section~4 of \cite{BZ}, where, for a normal $\cC$-open subgroup $K$ of $\hP^\cC$, 
we replace the homology group $H_1(K,\Q_p)$ with the homology group $H_1(\ol{S}_K,\Q_p)$ of the closed surface $\ol{S}_K$ associated to $K$ 
(cf.\ Section~\ref{linearization}). As in \cite{BZ}, Theorem~2.3 ibid.\ then implies:

\begin{theorem}\label{reidemeister}Let $\cC$ be a class of finite groups, $S$ a surface of negative Euler characteristic and let $\Pi=\pi_1(S,x)$
for some base point $x\in S$.
\begin{enumerate}
\item The free homotopy classes of a pair of elements $\alpha,\beta\in\Pi$ contain disjoint representatives
if and only if  ${\mathfrak R}_\cC(\alpha,\beta)=0$.
\item The free homotopy class of a nonpower element $\gm\in\Pi$ contains a
simple closed curve if and only if ${\mathfrak R}_\cC(\gm,\gm)=0$.
\end{enumerate}
\end{theorem}

\subsection{The set of simple elements is closed in the pro-$\cC$ topology}
In order to extend Proposition~4.1 \cite{BZ} to an arbitrary surface group $\Pi$, we have to restrict the group of automorphisms of $\hP^\cC$ to
those which are meaningful from a geometric point of view:

\begin{definition}\label{geo_auto}For every characteristic $\cC$-open subgroup $K$ of $\hP^\cC$, let $N_K$ be the kernel of the natural epimorphism 
$\hp_1(S_K,\td{x})^\cC\to\hp_1(\ol{S}_K,\td{x})^\cC$, where $\td{x}\in S_K$ is a base point lying over the base point $x\in S$ of $\Pi$ and we identify $K$ 
with the pro-$\cC$ completion $\hp_1(S_K,\td{x})^\cC$ of the fundamental group $\pi_1(S_K,\td{x})$. We say that an automorphism $f\in\aut(\hP^\cC)$ 
is \emph{coherent} if, for all subgroups $K$ of $\hP^\cC$ as above, $f$ preserves the subgroup $N_K$ (so that it induces 
an automorphism $\bar f_K$ of the quotient group $\hp_1(\ol{S}_K,\td{x})^\cC$). Let us then denote by $\aut^\mathrm{coh}(\hP^\cC)$ the subgroup of 
$\aut(\hP^\cC)$ consisting of coherent automorphisms. 
\end{definition}

\begin{remark}\label{subgroups}
Observe that $\aut^\mathrm{coh}(\hP^\cC)$ is a closed subgroup of $\aut(\hP^\cC)$ which contains the image of the pro-$\cC$ congruence 
mapping class group $\cGG(S,x)^\cC$ under the natural representation.
Moreover, if we identify $\hP^\cC$ with the geometric pro-$\cC$ algebraic fundamental group of a smooth projective curve $C$ defined over $\Q$ endowed 
with a rational point $x\in C$, $\aut^\mathrm{coh}(\hP^\cC)$ also contains the image of the absolute Galois group $G_\Q$ under the associated
Galois representation.
\end{remark}

The following, more familiar, characterization of coherent automorphisms is possibly of independent interest:

\begin{proposition}\label{coherent}Let $S=S_{g,n}$, for $2g-2+n>0$, and denote by $u_i\in\Pi$ a simple loop around the puncture on $S$
obtained removing $P_i$ from $S_g$, for $i=1,\ldots,n$. Then we have:
\[\aut^\mathrm{coh}(\hP^\cC)=\{f\in\aut(\hP^\cC)|\, f(u_i)\sim u_{\sg(i)}^{s_i}\mbox{ for some }\sg\in\Sg_n\mbox{ and }  s_i\in(\ZZ^\cC)^\ast, i=1,\ldots,n\}.\]
\end{proposition}

\begin{proof}This is a highly nontrivial but immediate consequence of Theorem~1.5 in \cite{BZ}.
\end{proof}

The same proof (word by word) of Proposition~4.1 in \cite{BZ} then gives:

\begin{proposition}\label{character}There is a character $\chi_p\co\aut^\mathrm{coh}(\hP^\cC)\to\Z_p^\ast$ such that, for all $f\in\aut^\mathrm{coh}(\hP^\cC)$ 
and $\alpha,\beta\in\hP^\cC$, we have  ${\mathfrak R}_\cC(f(\alpha),f(\beta))=\chi_p(f)\cdot{\mathfrak R}_\cC(\alpha,\beta)$.
For $f\in\Inn(\hP^\cC)$, we have $\chi_p(f)=1$ and so the character $\chi_p$ descends to a character $\bar\chi_p\co\out^\mathrm{coh}(\hP^\cC)\to\Z_p^\ast$.
\end{proposition}

Let us observe that, by the various universal properties involved, we can extend the pro-$\cC$ Reidemeister pairing to a bilinear pairing: 
$\Q_p[[\hP^\cC]]\otimes\Q_p[[\hP^\cC]]\to\Q_p[[\hP^\cC]]$, which is also continuous and invariant under the conjugacy action of $\hP^\cC$ 
on the domain. 

On the $\Q_p$-vector space $\Q_p[[\hP^\cC]]$, there is a natural
$\aut^\mathrm{coh}(\hP^\cC)$-module structure induced by the continuous action of $\aut^\mathrm{coh}(\hP^\cC)$ on $\hP^\cC$.
We let instead $\Q_p[[\hP^\cC]](1)$ be the $\aut^\mathrm{coh}(\hP^\cC)$-module defined letting the group $\aut^\mathrm{coh}(\hP^\cC)$ act on
the $\Q_p$-vector space $\Q_p[[\hP^\cC]]$ through the character $\chi_p$. By Proposition~\ref{character}, with these definitions, the bilinear 
pairing just defined gives rise to an $\aut^\mathrm{coh}(\hP^\cC)$-equivariant bilinear pairing:
\begin{equation}\label{bilinear_pairing}
{\mathfrak R}_\cC(\_,\_)\co\Q_p[[\hP^\cC]]\otimes\Q_p[[\hP^\cC]]\to\Q_p[[\hP^\cC]](1),
\end{equation}
which we call the \emph{pro-$\cC$ Reidemeister bilinear pairing}. Let us observe that, for $\cC=(p)$, if we give $\Q_p[[\hP^{(p)}]]$ the canonical
(pro) mixed Hodge structure defined by Hain in \cite{H1} and to $\Q_p[[\hP^\cC]](1)$ the Hodge structure of weight $-2$ obtained
tensoring the trivial Hodge structure on $\Q_p[[\hP^{(p)}]]$ by a Tate twist, the bilinear pairing ${\mathfrak R}_p(\_,\_)$ becomes a morphism of
(pro) mixed Hodge structures. These facts should be compared with Theorem~1 and Question~1.2 in \cite{H2} (cf.\ also Section~\ref{open}).

Proposition~\ref{character} and Theorem~\ref{reidemeister} then immediately imply:

\begin{corollary}\label{closed_orbit}For $\gm\in\Pi$ a simple element, we have 
\[\aut^\mathrm{coh}(\hP^\cC)\cdot\gm\cap\Pi=\GG(S,x)\cdot\gm,\] where we identify $\Pi$ with a subgroup of its pro-$\cC$ completion $\hP^\cC$.
In particular, the set of simple elements of $\Pi$ is closed in the pro-$\cC$ topology of $\Pi$.
\end{corollary}

\subsection{The $\cC$-congruence orbit of a simple element ($\ccP^\cC\equiv\hP^\cC$)}
In this section, we assume that $\ccP^\cC\equiv\hP^\cC$ (e.g.\ $\cC=\cF$ or $\cC=(p)$). In particular, by Corollary~\ref{faithful} and 
Corollary~\ref{compconf}, the $\cC$-congruence topology on the mapping class group $\GG(S,x)$ is defined unambigously. We then have:

\begin{corollary}\label{closed_orbit2}Let $\cC$ be a class of finite groups such that $\ccP^\cC\equiv\hP^\cC$ (e.g.\ $\cC=\cF$ or $\cC=(p)$) and let $U$ 
be an open subgroup of $\cGG(S,x)^\cC$ and $\gm\in\Pi$  a simple element. Then: 
\[U\cdot\gm\cap\Pi=(U\cap\GG(S,x))\cdot\gm,\] 
where we identify $\GG(S,x)$ with a subgroup of $\cGG(S,x)^\cC$. In particular, orbits of simple elements by $\cC$-congruence open subgroups of 
$\GG(S,x)$ are closed in the pro-$\cC$ topology of $\Pi$.
\end{corollary} 

\begin{proof}By Corollary~\ref{closed_orbit}, in order to prove Corollary~\ref{closed_orbit2}, it is enough to show that the stabilizer $\cGG^\cC_\gm$ of a 
simple element $\gm\in\Pi\subset\hP^\cC$ for the action of the pro-$\cC$ congruence mapping class group $\cGG(S,x)^\cC$ on $\hP^\cC$ is the closure 
of the stabilizer $\GG_\gm$ of $\gm$ for the action of $\GG(S,x)$ on $\Pi$. These stabilizer are, respectively, the centralizers in $\cGG(S,x)^\cC$ and 
$\GG(S,x)$ of the bounding pair map determined by $\gm$ on $S$. The claim then follows from Corollary~\ref{centralizers multitwists} and 
Theorem~\ref{stab pro-curves}.
\end{proof}

\subsection{Some open questions}\label{open}
In Corollary~\ref{closed_orbit}, we have seen that the intersection of the $\aut^\mathrm{coh}(\hP^\cC)$-orbit in $\hP^\cC$ of a simple element $\gm\in\Pi$ 
with $\Pi$ coincides with the $\GG(S,x)$-orbit of $\gm$. A natural question is then the following:

\begin{question}\label{orbit_closure}Does the closure in $\hP^\cC$ of the $\GG(S,x)$-orbit of a simple element $\gm\in\Pi$ coincides with the 
$\aut^\mathrm{coh}(\hP^\cC)$-orbit in $\hP^\cC$ of the same element?
\end{question}

An affirmative answer to this question would put strong restrictions on the profinite group $\aut^\mathrm{coh}(\hP^\cC)$ and I believe that it may
have important implications for Grothendieck-Teichm\"uller theory. Related to Question~\ref{orbit_closure} is the problem of characterizing the set of 
simple elements in $\hP^\cC$, which we define to be the closure in $\hP^\cC$ of the set of simple elements of $\Pi$.

Let us observe that the simple elements of $\hP^\cC$ fall in a finite number of $\cGG(S,x)^\cC$-orbits which are just the closures of the $\GG(S,x)$-orbits 
of simple elements in $\Pi$. In Theorem~4.2 of \cite{faith}, we showed that a procyclic subgroup of $\hP^\cC$ contains at most a simple element 
(the theorem is proved for $\cC=\cF$, but, as usual, it is not difficult to see that the same argument applies to any class of finite groups). 
We may then ask the question:

\begin{question}\label{orbit_closure2}
Is the set of simple elements of $\hP^\cC$ characterized by the property that a procyclic subgroup $C$ of $\hP^\cC$ 
contains a (unique) simple element, if and only if, this subgroup is primitive (i.e.\ it is not properly contained in some other
procyclic subgroup of $\hP^\cC$) and ${\mathfrak R}_\cC(\gm,\gm)=0$ for some (and then all) $\gm\in C$? 
\end{question}

I think that in order to advance in all the above questions it would be useful to understand the relation of the $p$-adic Reidemeister bilinear pairing
${\mathfrak R}_p(\_,\_)$ with the $p$-adic Goldman bracket introduced by Hain in \cite{H2} (cf.\ Section~1.2 ibid.).
It would also be interesting to see whether the Goldman bracket can be defined  for any class of finite groups $\cC$ (the case $\cC=\cF$ would be
particularly interesting). My impression is that, while the $p$-adic (or pro-$\cF$, if any) Goldman bracket is, from a technical point of view, 
more difficult than the $p$-adic (or pro-$\cF$) Reidemeister bilinear pairing, it also encodes more information.

\bigskip

\noindent Marco Boggi,\\ Departamento de Matem\'atica, UFMG, \\
Av. Ant\^onio Carlos, 6627 - Caixa Postal 702 \\ 
CEP 31270-901 - Belo Horizonte - MG, Brasil.
\\
E--mail:\,\,\, marco.boggi@gmail.com

\end{document}